\documentclass[final,reqno,onefignum,onetabnum]{siamart190516}
\usepackage{subdepth}
\usepackage[percent]{overpic}
\usepackage{amsmath,amstext,amsbsy,amssymb,mathdots}
\usepackage{booktabs,bm}
\usepackage{mathrsfs}
\usepackage{tikz,cite}
\usetikzlibrary{patterns}
\usepackage{tikz-qtree}
\usepackage{todonotes}
\usepackage{blindtext}
\usepackage{courier}
\usetikzlibrary{positioning}
\usetikzlibrary{shapes,arrows}
\usepackage[fleqn,tbtags]{mathtools}
\usepackage{amsfonts}
\usepackage{relsize}
\usepackage{xcolor,colortbl}
\usepackage{psfrag}
\usepackage{alltt}
\usepackage{arydshln}
\usepackage{enumerate}
\usepackage{enumitem}
\usepackage{multirow}
\usepackage{algorithm, algpseudocode}
\usepackage{framed}
\newcommand{\sep}{\mathsf{sep}}

\colorlet{lightgray}{gray!40}

\newtheorem{remark}[theorem]{Remark}

\newcommand{\erank}{{\rm rank}_\epsilon}

\newcommand{\C}{\mathbb{C}}
\newcommand{ \rank}{{\rm rank}}

\newcommand{\iu}{\mathrm{i}}
\newcommand{\nmin}{n_{\min}}

\newcolumntype{"}{@{\hskip\tabcolsep\vrule width 1pt\hskip\tabcolsep}}
\setlist[description]{font=\normalfont\space}

\title{Compression properties for large Toeplitz-like matrices} 
\author{Bernhard Beckermann\thanks{Universit\'{e} de Lille, Lille, France. (Bernhard.Beckermann@univ-lille.fr). Supported in part by EXPOWER (H2020 MSC 101008231), MOMENTUM (FWO), and Labex CEMPI (ANR-11-LABX-0007-01).} \and
Daniel Kressner\thanks{Institute for Mathematics, EPFL, Switzerland, (daniel.kressner@epfl.ch).} \and
Heather Wilber\thanks{Department of Applied Mathematics, University of Washington, Seattle, USA (hdw27@uw.edu). Supported by the National Science Foundation under Grant No. DMS-2410045.} }

\begin{document}
\maketitle

\begin{abstract} 
Toeplitz matrices are abundant in computational mathematics, and there is a rich literature on the development of fast and superfast algorithms for solving linear systems involving such matrices.
Any Toeplitz matrix can be transformed into a matrix with off-diagonal blocks that are of low numerical rank.
This compressibility is relied upon in practice in a number of superfast Toeplitz solvers. In this paper, we show that the compression properties of these matrices can be thoroughly explained using their displacement structure and Zolotarev numbers. We provide explicit bounds on the numerical ranks of important submatrices that arise when applying HSS, HODLR and other approximations with hierarchical low-rank structure to transformed Toeplitz and Toeplitz-like matrices. Our results lead to very efficient displacement-based compression strategies that can be used to formulate adaptive superfast rank-structured solvers. 
\end{abstract}



\section{Introduction}

A matrix that is constant along the diagonals is called a Toeplitz matrix:
\[
 T = \begin{bmatrix}
      t_0 & t_{-1} & \cdots & t_{-n+1} \\
      t_1 & t_0 & \ddots & \vdots  \\
      \vdots & \ddots & \ddots & t_{-1} \\
      t_{n-1} & \cdots & t_1 & t_0
     \end{bmatrix} \in \C^{n\times n}.
\]
This and similar matrix structures appear frequently in computational mathematics, for instance in the context of discretized differential and integral equations, orthogonal polynomials, time series analysis, as well as image and signal processing~\cite{antoulas2005approximation, grenander1958toeplitz, gray2006toeplitz, heinig2013algebraic, kailath1995displacement, pan2012structured}. Although the development of fast (and superfast) solvers for linear systems with Toeplitz matrices is a classical topic that is more than $60$ years old, 
there have still been significant advances in recent years. In particular, a novel and versatile class of solvers is based on the observation that any Toeplitz (or Toeplitz-like) matrix can be transformed into a Cauchy-like matrix $C$ with compressible off-diagonal blocks~\cite{chandrasekaran2007superfast, martinsson2005fast, xia2012superfast}. In turn, this allows one to leverage existing algorithms for matrices with hierarchical low-rank structure, not only for linear systems but also for, e.g., eigenvalue problems~\cite{Ou2022}, matrix functions, and matrix equations~\cite{KMR2019}.
The goal of this work is to perform an analysis of the compressibility of $C$.
Our analysis not only completes and improves upon existing analyses, but we will demonstrate that it is also the basis of new, deterministic algorithms for compressing $C$.

Most fast direct solvers for a matrix $T$ with Toeplitz and related structures are based on the observation that the displacement rank of $T$ is small. More specifically, in this work we consider the Sylvester displacement equation
\begin{equation} 
\label{eq:toepdisp}
Z T -TZ = GH^*, \quad Z = \begin{pmatrix} 0 & 1 \\ I_{n-1} & 0 \end{pmatrix}, \quad\rank({G}{H}^*) = \rho \ll n. 
\end{equation}
For a Toeplitz matrix $T$, one can choose the displacement rank $\rho=2$. 
Our developments also apply to Toeplitz-like matrices~\cite{kailath1995displacement,pan2012structured}, that is, matrices $T$ that satisfy~\eqref{eq:toepdisp} for some $\rho$ that is still small but possibly larger than $2$. Let us stress that the choice of displacement equation is by no means unique~\cite{kailath1995displacement}.
Our choice~\eqref{eq:toepdisp} is simple but not suitable for all purposes because the linear operator $T\mapsto Z T -TZ$ is singular and, in turn, $T$ cannot be fully reconstructed from the  \textit{generator matrices} $G$ and $H$. Other choices have been made in the literature, using coefficient matrices different from $Z$ or a Stein equation instead of a Sylvester equation.

Classical fast solvers of complexity \smash{$\mathcal{O}(n^2)$}, such as the generalized Schur algorithm, exploit the observation that Schur complements and inverses preserve the displacement rank, which in turn allows one to perform certain operations on $T$, such as the LU factorization, entirely in terms of the generator matrices; see~\cite{kailath1995displacement} for a survey. Classical superfast solvers of complexity \smash{$\mathcal{O}(n \log^2 n)$} are obtained by combining this idea with a divide-and-conquer approach; see~\cite{Huckle1998} and the references therein.

The matrix $Z$ in~\eqref{eq:toepdisp} is circulant and thus easily diagonalized by a unitary discrete Fourier transform $F$. In turn, the transformed matrix \smash{$C = FTF^*$}  satisfies a displacement equation with diagonal coefficients. This observation has been used in~\cite{gohberg1995fast} to derive fast solvers with enhanced numerical stability; pivoting does not destroy the Cauchy-like structure of $C$. More recently, several works~\cite{chandrasekaran2007superfast, martinsson2005fast, xia2012superfast} have derived new superfast solvers from the fact that off-diagonal submatrices of $C$ are well-approximated by low-rank matrices. Algorithm~\ref{alg:generalsolver} describes the general framework of these solvers applied to a linear system
\begin{equation}
\label{eq:tls}
Tx = b, \quad T \in \mathbb{C}^{n \times n}, 
\end{equation}
with a Toeplitz-like matrix $T$  satisfying~\eqref{eq:toepdisp}.
\begin{algorithm}[t!]
\caption{Superfast solvers for a Toeplitz-like linear system~\eqref{eq:tls}.}
\label{alg:generalsolver}
\begin{algorithmic}[1]
\State Compute $\widetilde{b} = Fb$, $\widetilde{G} = FG$, $\widetilde{H} = FH$. 
\State Determine \smash{$\widetilde{C}$}, a hierarchical low-rank approximation of $C = FTF^*$ from its generators $\widetilde{G}, \widetilde{H}$. 
\State Solve \smash{$\widetilde{C} \widetilde{x} = \widetilde{b}$}. 
\State Compute \smash{$x = F^*\widetilde{x}.$} 
\end{algorithmic}
\end{algorithm}

Steps 1 and 4 of Algorithm~\ref{alg:generalsolver} can be easily accomplished in \smash{$\mathcal{O}(n \rho \log n)$} operations using fast Fourier transforms (FFTs). 
 For Step 2, several hierarchical approximations of $C$ are possible. This includes the relatively simple HODLR (Hierarchical Off-Diagonal Low-Rank) format, as well as the more involved SSS (sequentially semi-separable) and HSS (hierarchical semi-separable) formats used in~\cite{chandrasekaran2007superfast,martinsson2011fast, xia2012superfast}.
 
 In all these formats, the approximation \smash{$\widetilde{C}$} is constructed by recursively partitioning $C$ into successively smaller submatrices and replacing each submatrix by a low-rank approximation. HSS format features additional structure; the low-rank factors between different levels of the recursion tree are nested; see Section~\ref{sec:HSSintro}. 
Its use leads to particularly efficient \smash{$\mathcal{O}(n p^2)$} direct solvers for Step 3~\cite{chandrasekaran2006fast,xia2012superfast}, 
assuming that the rank of each low-rank approximant in \smash{$\widetilde{C}$ is bounded by $p$}.
The efficiency of Steps 2 and 3 depends on (i) the hierarchical low-rank format used, (ii) the compressibility of $C$,  (iii) the method of approximation used to find  \smash{$\widetilde{C}$}, and (iv) the type of solver used in Step 3. 
In this work, we focus on analysis related to (ii) and (iii).  

\subsection{Main results}

The primary novelty in our paper is the introduction of new bounds that more thoroughly explain the low-rank properties of the Cauchy-like matrix $C$. This mproves upon some results that currently exist in the literature, and it provides explicit bounds on the singular values of the relevant submatrices of $C$. We show that the alternating implicit direction (ADI) method can be used to construct low rank approximations with errors controlled by our bounds. This leads to an efficient and adaptive algorithm for constructing HODLR and HSS approximations to $C$ in Step 2 of Algorithm~\ref{alg:generalsolver}. 
The current state-of-the-art for solving~\eqref{eq:tls} given an arbitrary Toeplitz matrix is based on randomized linear algebra~\cite{xia2012superfast}. Our analysis more thoroughly explains the effectiveness of this method, and it also serves as the basis of competitive deterministic analogue methods. 

To describe the compressibility of $C$, we use the notion of numerical rank, also called $\epsilon$-rank:
\begin{definition} 
\label{def:erank}
Let \smash{$\sigma_1(X) \ge \sigma_2(X) \ge \cdots \ge \sigma_{\min\{m,n\}}(X) \ge 0 $} denote the singular values of a matrix \smash{$X \in \mathbb{C}^{m \times n}$}. Given $0< \epsilon < 1$, the $\epsilon$-rank of $X$  is  \smash{$\erank(X) =k$} for the smallest integer $k$ such that \smash{$\sigma_{k+1}(X) \leq \epsilon \|X\|_2$},
where $\| \cdot \|_2 = \sigma_1(X)$ denotes the spectral norm. 
\end{definition}
If  \smash{$\erank(X) \leq k$} then $X$ is well-approximated by a rank $k$ matrix \smash{$X^{(k)}$} in the sense that \smash{$\|X - X^{(k)}\|_2 \leq\epsilon \|X\|_2$}. 

While some theoretical justification on the rank structure of $C$ is given in~\cite{chandrasekaran2007superfast, martinsson2011fast}, these arguments are incomplete (see section~\ref{sec:bounds}). For example, they do not duly justify the use of HODLR, HSS, or other weakly-admissible~\cite{hackbusch2004hierarchical} hierarchical approximations to $C$.  These arguments also do not supply a priori estimates on the numerical ranks of submatrices of $C$. In contrast, we explicitly bound the numerical ranks of the off-diagonal submatrices of $C$, including those arising in weakly admissible hierarchical formats. This leads to adaptive hierarchical low-rank approximation schemes that automatically select appropriate approximation parameters based on a relative tolerance parameter $0 < \epsilon < 1$.
Our bounds show that generally, an off-diagonal submatrix $X$ of $C$ has $\epsilon$-rank of order $\mathcal{O}(\rho  \log n \log (1/\epsilon) )$.  The compressibility of submatrices sufficiently separated from the main diagonal of the matrix is even better; 
we prove that in this case, the bound is \smash{$\mathcal{O} (\rho \log \tfrac{1}{\epsilon})$}, completely independent of $n$ (see Theorem~\ref{thm:bounds}).

\subsection{Summary} The rest of this paper is organized as follows: 
In Section~\ref{sec:lowrankprop}, we review the essential ideas from~\cite{Beckermann2019} and then use them to derive new bounds that characterize of the rank structure of $C$ (Section~\ref{sec:bounds}). 
In sections~\ref{sec:HODLR} and~\ref{sec:HSSintro}, we use our bounds to formulate efficient displacement-based compression strategies for finding HODLR and HSS approximations to $C$. Finally, we provide some
preliminary numerical evidence in Section~\ref{sec:practicalsolver}. 

\section{Compressibility of Cauchy-like matrix $C$}
\label{sec:lowrankprop}

In~\cite{Beckermann2019}, the $\epsilon$-rank for certain matrices with displacement structure is bounded via solutions to certain rational approximation problems. We derive bounds on the off-diagonal blocks of the Cauchy-like matrix $C$ using a similar strategy. 
 
 \subsection{Displacement structure of $C$} 
 
 \label{sec:displacement}

 Starting with the displacement equation satisfied by $T$ in~\eqref{eq:toepdisp}, we first diagonalize $Z$ with a unitary discrete Fourier transform. Letting $\omega = \exp ( \iu \pi/n)$, we have that 
\begin{equation} 
\label{eq:F}
 FZF^* = D, \quad F = \left( \frac{\omega^{2jk}}{\sqrt{n}}\right)_{j, k = 0, \ldots, n-1}, 
\end{equation} 
where \smash{$D = {\rm diag}( \omega^0, \omega^2, \ldots, \omega^{2n-2})$}. Here and in the following, we adopt the convention of indexing matrix entries starting from 0. Inserting~\eqref{eq:F} into~\eqref{eq:toepdisp}, it follows that the transformed matrix \smash{$C = F T F^*$} satisfies the displacement equation 
\begin{equation} 
\label{eq:cauchylike-disp}
D C - C D = \widetilde{G}\widetilde{H}^*, \quad \widetilde{G} = FG , \quad \widetilde{H} = F H ,
\end{equation}
with the same displacement rank $\rho$.

Given the the right-hand side of~\eqref{eq:cauchylike-disp}, each \emph{off-diagonal} entry $c_{jk}$, $j\not = k$, of $C$ can be recovered by multiplying the corresponding off-diagonal entry of $\widetilde{G}\widetilde{H}^*$ with
$1/(\omega^{2j} - \omega^{2k})$. The latter can be viewed as the discretization of the function $f(x,y) = 1/(x-y)$ with $x,y$ both on the unit circle. When $|x-y|$ is large enough, a truncated Taylor expansion of $f$ could be used to construct good low-rank approximations to
submatrices of $C$ only containing indices $(j,k)$ for which $|\omega^{2j} - \omega^{2k}|$ is not small; see, e.g.,~\cite{Boerm2010}.
Intuitively, this suggests that submatrices of $C$ far from the main diagonal as well as from the top-right and bottom-left corners (due to the periodicity of $\omega^{2j}$) should be compressible.  To get sharper bounds, especially for regions close to the singularity of $f$, we appeal to arguments based on displacement structure instead of Taylor expansions. 

Due to the singularity of the linear operator $C \mapsto DC - CD$, the diagonal entries of $C$ cannot be recovered from the right-hand side of~\eqref{eq:cauchylike-disp}. These missing entries of $C$ can be cheaply computed from $T$ by exploiting the original displacement equation~\eqref{eq:toepdisp}. To see this, consider the operator $\mathcal{S}:  \mathbb{C}^{n \times n} \to \mathbb{C}^{n \times n}$ defined by $\mathcal{S}(T) := ZT-TZ$. By the definition of $Z$, the null space of $\mathcal{S}$ is the subspace of all cyclic $n \times n$ matrices. Letting 
$\Pi(T)$ denote the orthogonal projection of $T$ onto this null space, it follows that $F \Pi(T) F^*$ is diagonal. Its diagonal entries coincide with the diagonal entries of $C$, which can thus be computed by
applying an FFT to the first column of $\Pi(T)$. This first column is cheap to compute because the projection
$\Pi(T)$ is just the average along the cyclic diagonals of $T$. In particular, when $T$ is a Toeplitz matrix,
the $k$th entry of this vector is given by 
$ \left( (n-k)t_{k}  + k t_{-n+k}  \right)/n.$
In summary, only $\mathcal{O}(n \log n)$ operations are needed to compute the diagonal of $C$. Throughout the rest of the paper, we will therefore focus on computing the off-diagonal entries of $C$.

 \subsection{Singular value decay for matrices with displacement structure}
 \label{sec:zolotarev}
 Let us consider a general displacement equation of the form $AX - XB = Y$, where $A$ and $B$ are normal matrices with disjoint spectra $\Lambda(A)$ and $\Lambda(B)$. It is shown in~\cite{Beckermann2019} that 
\begin{equation} 
\label{eq:zolobound_svs}
\sigma_{k \rho + 1}(X) \leq Z_{k}(E_1, E_2) \|X\|_2,
\end{equation}
where $\rho = \rank(Y)$ and \smash{$E_1, E_2$} are disjoint subsets of the complex plane with \smash{$\Lambda(A) \subset E_1$}, \smash{$\Lambda(B) \subset E_2$}.
The $k$th \textit{Zolotarev number} \smash{$Z_{k}(E_1, E_2)$} is defined by 
\begin{equation} \label{eq:zolotarev}
 Z_k(E_1,E_2) := \inf_{r\in \mathcal{R}_{k,k}} \frac{\sup_{z \in E_1} |r(z)| }{\inf_{z \in E_2} |r(z)|},
\end{equation}
where $\mathcal{R}_{k,k}$ contains all rational functions with numerator and denominator of degree at most $k$.

When \smash{$E_1$} and \smash{$E_2$} are well-separated,
one expects that \smash{$Z_k(E_1,E_2)$} decays rapidly as $k$ increases and, in turn, the singular values of $X$ also decay rapidly, provided that $\rho$ is small.\footnote{Note that one can still establish rapid singular value decay for large $\rho$ when the singular values of $Y$ decay sufficiently fast~\cite{townsend2018singular}.} 
To gain meaningful insights, simple tight and non-asymptotic bounds for \smash{$Z_k(E_1, E_2)$} are needed, which are only explicitly available for a few special cases of $E_1,E_2$; see~\cite{Beckermann2019}. In~\cite{rubin2022bounding}, a bound for the Zolotarev number is given for rather general $E_1$, $E_2$, but this bound depends on an invariant parameter associated with the conformal modulus of $E_1$ and $E_2$ and on the behavior of Faber rational functions, which can be difficult to tightly control for finite $k$. In general, this parameter can only be computed numerically. As we will see below, the sets $E_1$ and $E_2$ relevant in this work can always be taken as subarcs on the unit circle, which will allow us to
formulate and analyze explicit bounds on $Z_k(E_1, E_2)$.

\subsection{Bounding the $\epsilon$-ranks of submatrices of $C$} 
\label{sec:bounds}

Clearly, the technique from Section~\ref{sec:zolotarev} cannot be applied directly to the Sylvester equation~\eqref{eq:cauchylike-disp}; its coefficients are identical and thus have spectra that are anything but disjoint. 

On the other hand, due to the diagonal structure of $D$, any submatrix of $C$ satisfies a displacement equation as well. For subsets $J$ and $K$ of $\{0,\ldots,n\!-\!1\}$, we let $C_{JK}$ denote the $|J|\times |K|$ submatrix of $C$ containing all entries $c_{jk}$ with $j\in J$, $k \in K$.
By~\eqref{eq:cauchylike-disp}, we have that
\begin{equation}
\label{eq:dispY}
 D_J C_{JK} - C_{JK} D_K = \widetilde{G}_J \widetilde{H}_K^*,
\end{equation} 
where the diagonal matrix \smash{$D_J$} contains the diagonal elements $\omega^{2j}$ for $j\in J$, $D_K$ is defined analogously, and 
$\widetilde{G}_J$, $\widetilde{H}_K$ contain the corresponding rows of 
$\widetilde{G}$, $\widetilde{H}$.

As long as $J \cap K = \emptyset$, the spectra of $D_J$, $D_K$ are disjoint and the equation~\eqref{eq:dispY} is nonsingular. To apply the techniques from Section~\ref{sec:zolotarev}, we require the stronger assumption that the spectra of $D_J$ and $D_K$ are contained in disjoint arcs and we need to measure the relative gap between these arcs, taking the periodicity of $\omega^{2j}$ into account. This motivates the following definition.
\begin{definition} \label{def:mlsubmatrix}
Given index sets $J, K \subset \{0,\ldots,n\!-\!1\}$ and integers $m \ge 1$, $\sep \ge 1$, we call $C_{JK}$ an \emph{$(m, \sep)$ submatrix of $C$} if 
\begin{align*}
 &\max\{ |k\!-\!k^\prime|: k, k^\prime \in K \} \le m\!-\!1, \\
 &\min\{ |j\!-\!k|: j \in J, k \in K\} \ge \sep, \quad \min\{ |j\!-\!k\!-\!n|: j \in J, k \in K\} \ge \sep.
\end{align*}
\end{definition}
To illustrate Definition~\ref{def:mlsubmatrix}, consider $K = \{0,\ldots,m\!-\!1\}$. The maximal admissible row index set of an $(m, \sep)$ submatrix $C_{JK}$ is $J = \{m,m+2,\ldots,n\!-\!1\}$ for $\sep = 1$ and $J = \{2m,2m\!+\!1,\ldots,n\!-m\!-\!1\}$ for $\sep = m\!+\!1$; see also Figure~\ref{fig:submatrices}.

  \begin{figure}
\centering
\begin{tikzpicture}[scale=0.34, every node/.style={scale=0.85}]
\hspace{3pt}
\draw[ fill = lightgray, fill opacity = .3](0, 0) rectangle(1,9); 
\draw[ fill = lightgray, fill opacity = .3](1,0) rectangle(2,8); 
\draw[ fill = lightgray, fill opacity = .3](2,0) rectangle(3,7); 
\draw[ fill = lightgray, fill opacity = .3](3,0) rectangle(4,6); 
\draw[ fill = lightgray, fill opacity = .3](4,0) rectangle(5,5); 
\draw[ fill = lightgray, fill opacity = .3](5,0) rectangle(6,4); 
\draw[ fill = lightgray, fill opacity = .3](6,0) rectangle(7,3); 
\draw[ fill = lightgray, fill opacity = .3](7,0) rectangle(8,2); 
\draw[ fill = lightgray, fill opacity = .3](8,0) rectangle(9,1); 
\draw[ fill = lightgray, fill opacity = .3](9,1) rectangle(10,10); 
\draw[ fill = lightgray, fill opacity = .3](8,2) rectangle(9,10); 
\draw[ fill = lightgray, fill opacity = .3](7,3) rectangle(8,10); 
\draw[ fill = lightgray, fill opacity = .3](6,4) rectangle(7,10); 
\draw[ fill = lightgray, fill opacity = .3](5,5) rectangle(6,10); 
\draw[ fill = lightgray, fill opacity = .3](4,6) rectangle(5,10); 
\draw[ fill = lightgray, fill opacity = .3](3,7) rectangle(4,10); 
\draw[ fill = lightgray, fill opacity = .3](2,8) rectangle(3,10); 
\draw[ fill = lightgray, fill opacity = .3](1,9) rectangle(2,10); 

\draw[fill= yellow, fill opacity = .2](0, 10) rectangle(1,9); 
\draw[fill= yellow, fill opacity = .2](1,9) rectangle(2,8); 
\draw[fill= yellow, fill opacity = .2](2,8) rectangle(3,7); 
\draw[fill= yellow, fill opacity = .2](3,7) rectangle(4,6); 
\draw[fill= yellow, fill opacity = .2](4,6) rectangle(5,5); 
\draw[fill= yellow, fill opacity = .2](5,5) rectangle(6,4); 
\draw[fill= yellow, fill opacity = .2](6,4) rectangle(7,3); 
\draw[fill= yellow, fill opacity = .2](7,3) rectangle(8,2); 
\draw[fill= yellow, fill opacity = .2](8,2) rectangle(9,1); 
\draw[fill= yellow, fill opacity = .2](9,1) rectangle(10,0); 
\node[] at (5, -1) {(a) $\sep = 1$}; 
\end{tikzpicture}$\, $
\begin{tikzpicture}[scale=0.34, every node/.style={scale=0.85}]
\hspace{4pt}
\draw[fill= lightgray, fill opacity = .3](0, 1) rectangle(1,8); 
\draw[fill= lightgray, fill opacity = .3](1,0) rectangle(2,7); 
\draw[fill= lightgray, fill opacity = .3](2,0) rectangle(3,6); 
\draw[fill= lightgray, fill opacity = .3](2,9) rectangle(3,10); 
\draw[fill= lightgray, fill opacity = .3](3,0) rectangle(4,5); 
\draw[fill= lightgray, fill opacity = .3](3,8) rectangle(4,10);
\draw[fill= lightgray, fill opacity = .3](4,0) rectangle(5,4); 
\draw[fill= lightgray, fill opacity = .3](4,7) rectangle(5,10); 
\draw[fill= lightgray, fill opacity = .3](5,0) rectangle(6,3); 
\draw[fill= lightgray, fill opacity = .3](5,6) rectangle(6,10); 
\draw[fill= lightgray, fill opacity = .3](6,0) rectangle(7,2); 
\draw[fill= lightgray, fill opacity = .3](6,5) rectangle(7,10); 
\draw[fill= lightgray, fill opacity = .3](7,0) rectangle(8,1); 
\draw[fill= lightgray, fill opacity = .3](7,3) rectangle(8,10); 
\draw[fill= lightgray, fill opacity = .3](8,3) rectangle(9,10); 
\draw[fill= lightgray, fill opacity = .3](9,2) rectangle(10,9); 
\draw[fill= yellow, fill opacity = .2](0, 10) rectangle(1,9); 
\draw[fill= yellow, fill opacity = .2](1,9) rectangle(2,8); 
\draw[fill= yellow, fill opacity = .2](2,8) rectangle(3,7); 
\draw[fill= yellow, fill opacity = .2](3,7) rectangle(4,6); 
\draw[fill= yellow, fill opacity = .2](4,6) rectangle(5,5); 
\draw[fill= yellow, fill opacity = .2](5,5) rectangle(6,4); 
\draw[fill= yellow, fill opacity = .2](6,4) rectangle(7,3); 
\draw[fill= yellow, fill opacity = .2](7,3) rectangle(8,2); 
\draw[fill= yellow, fill opacity = .2](8,2) rectangle(9,1); 
\draw[fill= yellow, fill opacity = .2](9,1) rectangle(10,0); 
\draw[fill= blue, fill opacity = .2](0,1) rectangle(1,0); 
\draw[fill= blue, fill opacity = .2](0,9) rectangle(1,8); 
\draw[fill= blue, fill opacity = .2](1,8) rectangle(2,7); 
\draw[fill= blue, fill opacity = .2](2,7) rectangle(3,6); 
\draw[fill= blue, fill opacity = .2](3,6) rectangle(4,5); 
\draw[fill= blue, fill opacity = .2](4,5) rectangle(5,4); 
\draw[fill= blue, fill opacity = .2](5,4) rectangle(6,3); 
\draw[fill= blue, fill opacity = .2](6,3) rectangle(7,2); 
\draw[fill= blue, fill opacity = .2](7,2) rectangle(8,1); 
\draw[fill= blue, fill opacity = .2](8,1) rectangle(9,0); 
\draw[fill= blue, fill opacity = .2](9,10) rectangle(10,9); 
\draw[fill= blue, fill opacity = .2](1,10) rectangle(2,9); 
\draw[fill= blue, fill opacity = .2](2,9) rectangle(3,8); 
\draw[fill= blue, fill opacity = .2](3,8) rectangle(4,7); 
\draw[fill= blue, fill opacity = .2](4,7) rectangle(5,6); 
\draw[fill= blue, fill opacity = .2](5,6) rectangle(6,5); 
\draw[fill= blue, fill opacity = .2](6,5) rectangle(7,4); 
\draw[fill= blue, fill opacity = .2](7,4) rectangle(8,3); 
\draw[fill= blue, fill opacity = .2](8,3) rectangle(9,2); 
\draw[fill= blue, fill opacity = .2](9,2) rectangle(10,1); 
\node[] at (5, -1) {(b) $\sep = m+1$}; 
\end{tikzpicture}
\caption{Definition~\ref{def:mlsubmatrix} applied to an $n\times n$ matrix partitioned into block columns of width $m$. The largest $(m,\sep)$ submatrix in each block column is shown in gray. (a) For $\sep =1$ only the diagonal blocks (yellow) are not admissible. 
(b) For $\sep =m+1$, also the super/subdiagonal blocks as well as the bottom left/top right blocks (blue) are not admissible.
\label{fig:submatrices}}
\end{figure}

The ``pinched'' block columns displayed in Figure~\ref{fig:submatrices} (a) for $\sep = 1$ correspond to the notion of HSS block columns that play a central role in the compression in the HSS format. The notion of HSS block rows is obtained by applying Definition~\eqref{def:mlsubmatrix} to $C^*$. In the literature on hierarchical and $\mathcal H^2$ matrices~\cite{hackbusch2004hierarchical}, the case $\sep = 1$ corresponds to the 
weak admissibility criterion, while $\sep \ge m+1$ loosely corresponds to the strong admissibility criterion.  

\begin{theorem} 
\label{thm:bounds}
Consider integers $\sep \geq 1$, $m \geq 2$ such that $n \geq 2(m  +  \sep   -  1)$. Let $C_{JK}$ be an $(m,\sep)$ submatrix of a matrix $C$ satisfying~\eqref{eq:cauchylike-disp} with \smash{$\rank(\widetilde{G}\widetilde{H}^*) \leq \rho$}.  Then, 
\begin{equation}
\label{eq:erank_bnd}
\erank(C_{JK}) \leq \rho \left \lceil \frac{2}{\pi^2} \log\left( 4 \frac{ m + \sep \!-\!1}{\sep} \right) \log \left( \frac{4}{\epsilon} \right)  \right \rceil, 
\end{equation}
\end{theorem}

The bound~\eqref{eq:erank_bnd} implies that there are universal constants $C_1,C_2$ (independent of $J,K,m,n,\rho,\epsilon$)
such that $\erank(C_{JK})$ is bounded by $C_1 \rho \log m \log \epsilon^{-1}$ in the general, weakly admissible case ($\sep \ge 1$)
and by $C_2 \rho \log \epsilon^{-1}$ in the strongly admissible case ($\sep  \ge m + 1$).
\begin{proof}[Proof of Theorem~\ref{thm:bounds}]
By assumption, $K \subset [j,j\!+\!m\!-\!1]$ for some $k \ge 0$.
Let us recall that $C_{JK}$ satisfies the displacement equation~\eqref{eq:dispY}. By multiplying both sides of this equation with $\omega^{-2j-m+1}$,

we may assume without loss of generality that the spectra of $D_K$ and $D_J$ are contained within the following arcs of the unit circle:
\begin{equation*} 
 \Lambda(D_J) \subset \mathcal{A}_J := \{  e^{\iu t}: t \in [ \beta, 2 \pi \!- \!\beta]\},\ \Lambda(D_K) \subset \mathcal{A}_K := \{ e^{\iu t} \, :\, t \in [-\alpha, \alpha]\},
\end{equation*}
with 
\begin{equation} 
\label{eq:tau_kappa}
0 < \alpha = \frac{\pi}{n}(m\!-\!1) < \beta = \frac{\pi}{n}(m \!-\! 1\! + \!2 \cdot \sep) \leq \pi \!-\! \alpha.
\end{equation}
By~\eqref{eq:zolobound_svs}, we have that
\begin{equation} \label{eq:zolonew}
 \sigma_{\rho k + 1}(C_{JK}) \le Z_k(\mathcal{A}_J,\mathcal{A}_K) \|C_{JK}\|_2.
\end{equation}
To bound $Z_k( \mathcal{A}_J,\mathcal{A}_K)$ we observe that the M\"obius
map $T(z)= \mathrm i \tan(\beta/2) \frac{z+1}{z-1}$ maps \smash{$\mathcal{A}_J \cup \mathcal{A}_K$} onto the union of two closed disjoint real intervals:
\[
T(\mathcal A_J) = [-1,1],\quad T(\mathcal A_K) = \big\{x\in \mathbb R\colon |x| \ge 1/|\kappa|\big\}, \quad \kappa = \frac{\tan(\alpha/2)}{\tan(\beta/2)} \in (0,1).
\]
Since Zolotarev numbers are invariant under M\"obius transformations~\cite{Beckermann2019}, we have that
$
 Z_k(\mathcal{A}_J, \mathcal{A}_K) = Z_k\big(T_0(T(\mathcal A_J)),T_0(T(\mathcal A_k))\big)$ with a suitable M\"obius map $T_0$ such that $T_0(T(\mathcal A_J))$ and $T_0(T(\mathcal A_k))$ are real disjoint \emph{compact} intervals.

This allows us to use known results on Zolotarev numbers for intervals~\cite[Cor. 4.2]{Beckermann2019}: 
\begin{equation*} \label{eq:zolo_bnd}
Z_k\big(T_0(T(\mathcal A_J)),T_0(T(\mathcal A_k))\big) \leq 4 \rho_0^{-2k}, \quad \rho_0 = \exp\Big( \frac{\pi^2}{4 \mu (1/\sqrt{\gamma})}\Big) > 1 , 
\end{equation*}
where $\mu(\cdot)$ is the \textit{Gr\"otzsch ring function}, which is related to the complete elliptic integrals of the first kind~\cite[Sec.~3]{Beckermann2019},
and $\gamma$ is the cross-ratio of the endpoints of $T_0(T(\mathcal A_J))$ and of $T_0(T(\mathcal A_k))$. Since the cross-ratio of four points in the complex plane is invariant under M\"obius transformations, we obtain the following expressions for our cross-ratio $\gamma$:
\begin{equation} 
\label{eq:cr_arcsnew}
\gamma 
= \left(\frac{1+\kappa}{1-\kappa} \right)^2
= \left( \frac{ \sin \left( \tfrac{\beta + \alpha}{2} \right)}{\sin \left( \tfrac{\beta - \alpha}{2} \right)} \right)^2.
\end{equation}
Because of \smash{$\mu(1 / \sqrt{\gamma})  \leq \log( 4 \sqrt{\gamma})$}~\cite[19.9.5]{olver2010nist}, we obtain
 \begin{align} \label{eq:auxinequality}
\mu(1 / \sqrt{ \gamma} ) & 
\leq \log\left( 4 \frac{1+\kappa}{1-\kappa} \right) \le \log \left( 4 \frac{ \beta + \alpha}{\beta - \alpha} \right)=
\log \left( 4 \frac{ m+\sep-1}{\sep} \right),
 \end{align}
where the second inequality uses the concavity of the sine function.
 In summary,
 \begin{equation}
\label{eq:bound_zolo_arcs}
 Z_k(\mathcal{A}_J, \mathcal{A}_K)  \leq 4 \xi^{-k}, \quad \xi = \exp \left(  \pi^2 / \left( 2 \log \left( 4 \frac{ m+\sep-1}{\sep} \right)  \right) \right). 
\end{equation}
 Combined with~\eqref{eq:zolonew} this proves~\eqref{eq:erank_bnd}. 
\end{proof}

\begin{figure} 
 \centering
 \vspace{-.5cm}
  \begin{overpic}[width=.35\textwidth]{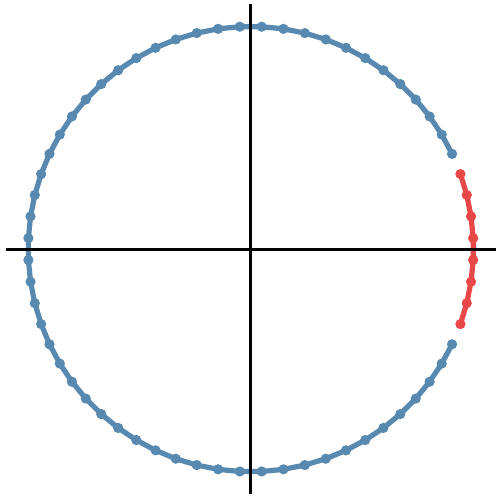} 
  \put(46,101) {$Im$}
  \put(-10, 48) {$Re$}
  \put(17, 71){$\mathcal{A}_J$}
  \put(78, 58){$\mathcal{A}_K$}
  \end{overpic}
 \centering
 \hspace{1cm}
  \begin{overpic}[width=.4\textwidth]{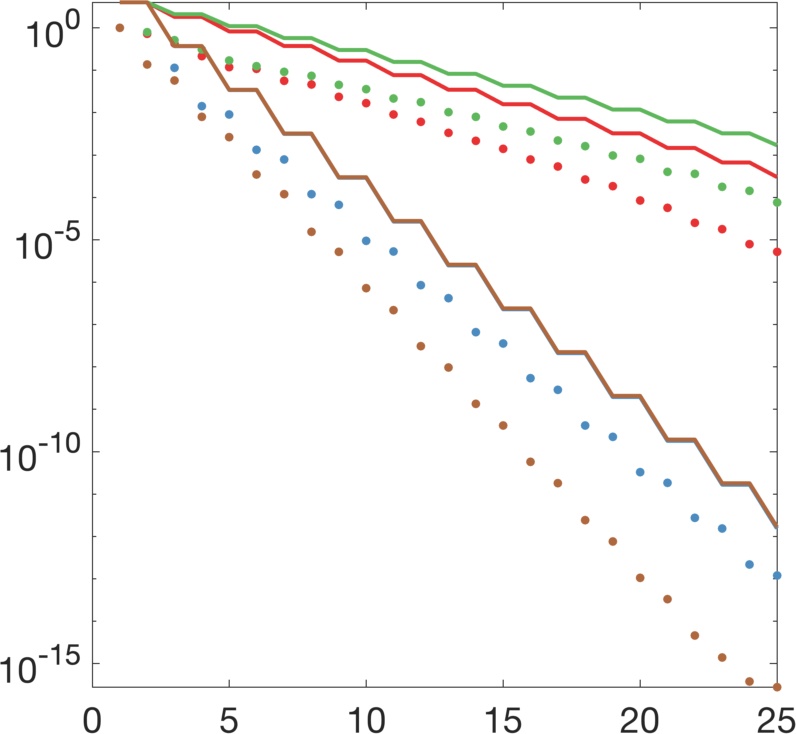}
  \put(45, -5){ \small index}
  \put(-7, 35){\rotatebox{90}{ \small magnitude}}
  \end{overpic}
  \caption{Left: Spectra \smash{$\Lambda(D_J)$} (blue dots) and \smash{$\Lambda(D_K)$} (red dots) in~\eqref{eq:dispY} for \smash{$J = \{ m , m\!+\!1, \ldots, n\!-\!1\}$, $K = \{0, 1, \ldots, m\!-\!1\}$} with $n = 64, m = 8,\sep = 1$. 
  The bounds of Theorem~\ref{thm:bounds} on \smash{$\erank(C_{JK})$} depend on the cross ratio of the end points of the arcs \smash{$\mathcal{A} _J\supset \Lambda(D_J)$} and \smash{$\mathcal{A}_K \supset \Lambda(D_K)$}.
 Right: Normalized singular values (dots) for various $(m, \sep)$ submatrices of $C$ satisfying~\eqref{eq:cauchylike-disp} for $n = 2048$, $\rho = 2$, and randomly chosen $B,C$. Singular value bounds (solid lines) according to~\eqref{eq:bound_zolo_arcs}. Red: $(m = 128, \sep = 1)$, green: $(m = 512, \sep = 1)$, blue: $(m = 128, \sep = 129)$, brown: $(m = 512, \sep = 513)$.}
  \label{fig:arcs}
  \end{figure}
  
  Figure~\ref{fig:arcs} illustrates the result of Theorem~\ref{thm:bounds} and its proof. The singular value bound~\eqref{eq:bound_zolo_arcs} matches the actual singular values fairly well. In both cases, when $\sep = 1$ and $\sep = m+1$, the singular values decay exponentially fast, but when $\sep = m+1$, the singular values decay at a much faster rate and, in turn, the $\epsilon$-rank grows more slowly as $\epsilon$ decreases.
  
\begin{remark}
In the theory of hierarchical matrices~\cite[Sec 5.2.2]{hackbusch2015hierarchical}, admissibility is often quantified by the relative gap 
\[
 \mathsf{gap}(E_1,E_2) = \frac{\mathsf{dist}(E_1,E_2)}{\min(\mathsf{diam}(E_1),\mathsf{diam}(E_2))}
\]
for disjoint $E_1,E_2 \subset \mathbb C$, where $\mathsf{diam}$ denotes the diameter of a set and $\mathsf{dist}$ denotes the
distance between two sets in $\mathbb C$. With the notation from the proof of Theorem~\ref{thm:bounds}, one has that
\[
 g:=\mathsf{gap}( T(\mathcal{A}_J),T(\mathcal{A}_K)) = \frac12 \Big( \frac{1}{\kappa} - 1 \Big).
\]
Using~\eqref{eq:auxinequality}, it follows
that we get a (small) improvement of~\eqref{eq:bound_zolo_arcs} by expressing our findings in terms of the relative gap $g$, that is,~\eqref{eq:bound_zolo_arcs} remains true when replacing the fraction $(m+\sep-1)/\sep$ by the smaller quantity
$$
    \frac{1+g}{g} = \frac{1+\kappa}{1-\kappa} .
$$
\end{remark}

\subsubsection{Related results}

Most of the existing bounds in the literature on the numerical rank of Cauchy matrices are derived from the separable approximation of the Cauchy kernel~\cite{chandrasekaran2007superfast,Penzl2000,Tyrtyshnikov1996}: Suppose there exist $2k$ functions $f_j,g_j$ and $0 < \epsilon_k < 1$ such that
\begin{equation} \label{separablexp}
  \Big| \frac{1}{x-y} - \sum_{j = 1}^k f_j(x) g_j(y)  \Big| \le \epsilon_k \Big| \frac{1}{x-y} \Big|, \quad \forall x \in E_1, y \in E_2,
\end{equation}
where $E_1, E_2$ are either real disjoint intervals or disjoint circular arcs. This implies that there is a rank-$k$ matrix $C_k$ such that 
the inequality $|C-C_k| \le \epsilon_k |C|$ holds elementwise. For instance, Braess and Hackbusch~\cite{braess2009efficient} obtained~\eqref{separablexp} for real intervals through approximating the function $1/z$ by exponential sums. For a rational function $r\in \mathcal{R}_{k,k}$, it is not difficult
to verify that we may write
\[
 \frac{1}{x-y} \Big( \frac{r(x)}{r(y)}-1 \Big) = \sum_{j = 1}^k f_j(x) g_j(y)
\]
for suitable $f_j,g_j$. Hence,~\eqref{separablexp} holds provided that
\begin{equation} \label{zololight}
 \max_{x \in E_1} |r(x)| \big/ \min_{y \in E_2} |r(x)| = \epsilon_k.
\end{equation}
For real intervals $E_1,E_2$, this approach for a not necessarily optimal $r$ has been suggested by Tyrtyshnikov~\cite{Tyrtyshnikov1996}; see also related work by Penzl~\cite{Penzl2000},~\cite{Grasedyck2004}, and Oseledets~\cite{Oseledets2007}.
Dutt and Rokhlin~\cite[Sec. 2.2]{dutt1995fast} observed that an optimal \emph{polynomial} $r$ of degree at most $k$ in~\eqref{zololight} is given by the scaled Chebyshev polynomial $T_k$. They also suggest a change of variables in order to treat the case of circular arcs $E_1,E_2$. Lepilov and Xia~\cite[Theorem 1]{Lepilov2024} solve for circular arcs the extremal problem~\eqref{zololight} for the class $r(z) = (z-c)^k$ with optimized choice of $c \in \mathbb C$, which is related to the truncated Taylor expansion at $c$ of the Cauchy kernel. In terms of the angles $0 < \alpha < \beta < \pi$ defined in~\eqref{eq:tau_kappa}, the approximation rates $\epsilon_k$ in~\eqref{zololight} for circular arcs established in~\cite{dutt1995fast}, cited in~\cite{martinsson2005fast},
and~\cite{Lepilov2024} are given by
\[
 \epsilon^{\text{\cite{dutt1995fast}}}_k = 1 \Big/ T_k\Big( \frac{\tan(\beta/2)}{\tan(\alpha/2)} \Big) \quad \text{and} \quad 
 \epsilon^{\text{\cite{Lepilov2024}}}_k = \Big(\frac{\sin(\alpha/2)}{\sin(\beta/2)}  \Big)^k,
\]
respectively. In the strongly admissible case ($\sep \ge m+1$ or, equivalently, $\alpha \le \beta/3$), this leads to the upper bounds
$1/T_k(3)$ and $3^{-k}$, respectively, implying that the numerical rank depends logarithmically on $1/\epsilon$. However, neither~\cite{dutt1995fast,martinsson2005fast}
nor~\cite{Lepilov2024} provide effective approximation rates for the weakly admissible case $\sep = 1$, which is of major importance in popular hierarchical low-rank formats. Indeed, when $\sep = 1$, one gets $\epsilon^{\text{\cite{dutt1995fast}}}_{\lceil \sqrt{m} \rceil} \to 2 e^{-2}$ and $\epsilon^{\text{\cite{Lepilov2024}}}_m \to e^{-2}$, provided that $m,n \to \infty$ such that $m/n \to 0$.

\subsubsection{Poles and zeros of Zolotarev rational functions}
\label{sec:poleszeros}

A rational function \hbox{$r_k \in \mathcal{R}_{k,k}$} that attains the infimum for $Z_k$ in~\eqref{eq:zolotarev} is called a Zolotarev rational function. The poles and zeros of~$r_k$ turn out to be useful parameters in a variety of iterative methods for solving Sylvester matrix equations~\cite{benner2009adi,druskin2011analysis,li2002low, wachspress1962optimum}. The poles and zeros of the Zolotarev rational function for 
$Z_k(\mathcal{A}_J, \mathcal{A}_k)$
admit explicit expressions in terms of elliptic integrals and, as we will see below 
in Section~\ref{sec:HODLR}, they can be used to construct low-rank approximations to $C_{JK}$. 
To specify these expressions, we use a M\"obius transformation $T_1$ to map from 
$[-\delta, -1] \cup [1, \delta]$ to $\mathcal A_J \cup \mathcal A_K$, where 
$$\delta =-1 + 2 \gamma + 2 \sqrt{\gamma^2-\gamma} > 1,$$
and $\gamma$ is the modulus of the cross ratio given by~\eqref{eq:cr_arcsnew}; see also~\cite[Eq. (3.7)]{Beckermann2019}.
This matches the construction in the proof of \cref{thm:bounds}; in fact, $T_1$ coincides 
with the inverse of $T_0 \circ T$ for a suitable choice of $T_0$.
 \begin{lemma}
\label{lemma:Shift_parameters}
With the notation introduced above, the Zolotarev rational function $r_k$ for $Z_k(\mathcal{A}_J, \mathcal{A}_K)$ has  zeros $\tau_1, \cdots, \tau_k$  and  poles $\nu_1, \cdots, \nu_k$ given by
$$ \tau_j = T_1\left( -\delta \, {\rm dn}\left[ \frac{2j - 1}{2k} K(\Xi), \Xi \right] \right), \quad \nu_j = T_1\left( \delta \, {\rm dn}\left[ \frac{2j - 1}{2k} K(\Xi), \Xi \right] \right), $$
where $\Xi = \sqrt{1-1/\gamma^2}$, $K(\cdot)$ is the complete elliptic integral of the first kind and ${\rm dn}[\cdot, \cdot]$ is the Jacobi elliptic function of the third kind.
\end{lemma}
\begin{proof}
This result immediately follows from the solution of Zolotarev's third problem on $[-\delta, -1] \cup [1, \delta]$
(see, e.g.,~\cite{Akhiezer1990} or~\cite[Theorem 2.1]{Fortunato2020}) and the invariance of $Z_k$ under M\"obius transformations.
\end{proof}
While a suitable $T_1$ is easy to derive and implement, some attention
is needed to 
accurately evaluate the elliptic functions in~\cref{lemma:Shift_parameters}
when the gap between $\mathcal{A}_J$ and $\mathcal{A}_K$ is small. We found that the standard \texttt{ellipk} and \texttt{ellipj} commands in MATLAB are insufficient and instead use an implementation found in the Schwarz-Christoffel Toolbox~\cite{driscoll1994schwarz}.

\section{HODLR approximation of Cauchy-like matrix $C$}
\label{sec:HODLR}

Theorem~\ref{thm:bounds} implies that the off-diagonal blocks of $C$ are of low numerical rank. 
Hierarchical low-rank formats~\cite{hackbusch2015hierarchical} are designed to take advantage of this property. In this section, we describe an algorithm that exploits the results from Section~\ref{sec:lowrankprop} 
to approximate $C$ with a matrix possessing the HODLR (hierarchical off-diagonal low rank) format~\cite{martinsson2011fast}. It is one of the most basic hierarchical structures and has the advantage of being relatively simple to describe and implement. 
To further simplify the description, we will assume that $n$ is a power of two throughout the rest of the paper.

\begin{figure}
\centering
\begin{tikzpicture}[scale=0.5, every node/.style={scale=0.8}]
\draw[black, thick](0,0)--(8,0);
\draw[black, thick] (0,0)--(0,8);
\draw[black, thick] (8,0)--(8,8);
\draw[black, thick] (0,8)--(8,8);
\draw[black, thick](4,0)--(4,8);
\draw[black, thick](0,4)--(8,4);
\draw[black, thick](2, 4)--(2, 8); 
\draw[black, thick](0, 6)--(4, 6); 
\draw[black, thick](6, 0)--(6, 4); 
\draw[black, thick](4, 2)--(8,2);
\draw[black, thick](0, 7)--(2, 7); 
\draw[black, thick](1,6)--(1,8);
\draw[black, thick](2, 5)--(4,5); 
\draw[black, thick](3, 4)--(3, 6); 
\draw[black, thick](4, 3)--(6,3); 
\draw[black, thick](5,2)--(5, 4); 
\draw[black, thick](6,1)--(8,1); 
\draw[black, thick](7, 0)--(7, 2); 
\node[] at (6,6) (a) {$C_1$};
\node[] at (2,2) (a) {$C_2$};
\node[] at (3,7) (a) {$C_3$};
\node[] at (1,5) (a) {$C_4$}; 
\node[] at (7,3) (a) {$C_5$};
\node[] at (5,1) (a) {$C_6$}; 
\node[] at (1.5, 7.5) (a) {$C_7$}; 
\node[] at (0.5, 6.5) (a) {$C_8$}; 
\node[] at (3.5, 5.5) (a) {$C_9$}; 
\node[] at (2.5, 4.5) (a) {$C_{10}$};
\node[] at (5.5, 3.5) (a) {$C_{11}$};
\node[] at (4.5, 2.5) (a) {$C_{12}$}; 
\node[] at (7.5, 1.5) (a) {$C_{13}$}; 
\node[] at (6.5, 0.5) (a) {$C_{14}$}; 
\draw [pattern=north west lines, pattern color=blue] (0,7) rectangle (1,8);
\draw [pattern=north west lines, pattern color=blue] (1,6) rectangle (2,7);
\draw [pattern=north west lines, pattern color=blue] (2,5) rectangle (3,6);
\draw [pattern=north west lines, pattern color=blue] (3,4) rectangle (4,5);
\draw [pattern=north west lines, pattern color=blue] (4,3) rectangle (5,4);
\draw [pattern=north west lines, pattern color=blue] (5,2) rectangle (6,3);
\draw [pattern=north west lines, pattern color=blue] (6,1) rectangle (7,2);
\draw [pattern=north west lines, pattern color=blue] (7,0) rectangle (8,1);
\end{tikzpicture}
\begin{tikzpicture}[sibling distance=5.5pt,scale=0.9, every node/.style={scale=0.9}]
\tikzset{level 1/.style={level distance=35pt}}
\tikzset{level 2/.style={level distance=40pt}}
\tikzset{level 3+/.style={level distance=40pt}}
\hspace{.5 cm}  \Tree [ .\node[]{$0$ };   
										[ .\node[]{$1$ }; 
											[ .\node[]{$3$ };
												[.\node[]{$7$};]
												[.\node[]{$8$};]
											]
											[ .\node[]{$4$};
												[.\node[]{$9$};]
												[.\node[]{$10$};]
											]
										] 
									    [ .\node[]{$2$}; 
									    	[ .\node[]{$5$};
									    		[.\node[]{$11$};]
												[.\node[]{$12$};]
									    	]
											[ .\node[]{$6$};
												[.\node[]{$13$};]
												[.\node[]{$14$};]
											]
										] 
									]

\end{tikzpicture}
\begin{tikzpicture}[scale=0.9, every node/.style={scale=0.9}]
\hspace{1cm}\node[] at (3,8.5) {(level $0$)};
\node[] at (3,7.3) {(level $1$)};
\node[] at (3,6) {(level $2$)};
\node[] at (3,4.6) {(level $3$)};
\end{tikzpicture}
\caption{A HODLR partitioning of $C$ is found by tessellating $C$ according to a binary tree \smash{$\mathcal{T}$} (right) of depth $3$.   Every labeled block corresponds to a vertex on the tree, with the root vertex (level $\ell = 0$) corresponding to $C$ as a whole. Each labeled block is (approximately) represented in factorized low rank form. On the lowest level, the diagonal blocks (filled blue) are stored explicitly.}
\label{fig:HODLRandTree}
\end{figure}

A matrix \smash{$\widetilde{C} \in \mathbb{C}^{n \times n}$} is called a HODLR matrix if it can be partitioned into equal-sized blocks
\begin{equation} 
\label{eq:HODLR}
\widetilde{C} = \begin{bmatrix} \widetilde{C}_{11} & \widetilde{C}_{12} \\ \widetilde{C}_{21} & \widetilde{C}_{22} \end{bmatrix}, 
\end{equation} 
such that \smash{$\widetilde{C}_{12}$, $\widetilde{C}_{21}$} are low-rank matrices represented in factorized form, and \smash{$\widetilde{C}_{11}$, $\widetilde{C}_{22}$} again admit a partitioning of the form~\eqref{eq:HODLR}. This partitioning is continued recursively until a prescribed minimum block size $\nmin$ is reached.
The definition~\eqref{eq:HODLR} induces a recursive block structure on an arbitrary $n\times n$ matrix, which will be called HODLR partitioning; see Figure~\ref{fig:HODLRandTree}. The maximum rank of any off-diagonal block appearing in the recursive construction is called the HODLR rank of $\widetilde{C}$.

It is convenient to associate a HODLR partitioning with a perfectly balanced binary tree \smash{$\mathcal{T}$}. We number the vertices of the tree consecutively, level by level from top to bottom, as shown in Figure~\ref{fig:HODLRandTree}, so that level $\ell$  contains the vertices \smash{$2^\ell\! -\! 1, \ldots, 2^{\ell+1}\! - \!2$}. Each vertex $v$ of \smash{$\mathcal{T}$} then has children $2v+1$ and $2v+2$. As shown in Figure~\ref{fig:HODLRandTree}, each $v$ is associated with an off-diagonal block \smash{${C}_v$} from the HODLR partition of ${C}$.  This block is defined in terms of its row indices as follows:
Set \smash{$J_0 = \{ 0, \ldots, n\!-\!1\}$}. For each parent vertex $v$ with associated index set \smash{$J_v$} of length $m$, the index sets of the children $2v+1$ and $2v+2$ are defined recursively as \smash{$J_{2v\!+\!1} = \{ (J_v)_1, \ldots (J_v)_{m/2}\}$} and $J_{2v\!+\!2} = J_v \setminus J_{2v\!+\!1}$. At a vertex $v$, the corresponding off-diagonal block is then given by ${C}_{v} = {C}(J_v, J_{\widetilde{v}})$, where \smash{$\widetilde{v}$} is the sibling vertex of $v$ in $\mathcal{T}$ (that is, $\widetilde v = v \!-\!1$ if $v$ is even and $\widetilde v = v\!+\!1$ if $v$ is odd).  Note that $J_v \cap J_{\widetilde{v}} = \emptyset$ by construction. When $C$ is approximated by a HODLR matrix each \emph{off-diagonal HODLR block} ${C}_{v}$ is approximated by a low-rank matrix and stored in factored form.

\subsection{Displacement-based approximation of off-diagonal blocks}
\label{sec:HODLR_factorization}

\begin{figure} 
 \centering
  \begin{overpic}[width=.44\textwidth]{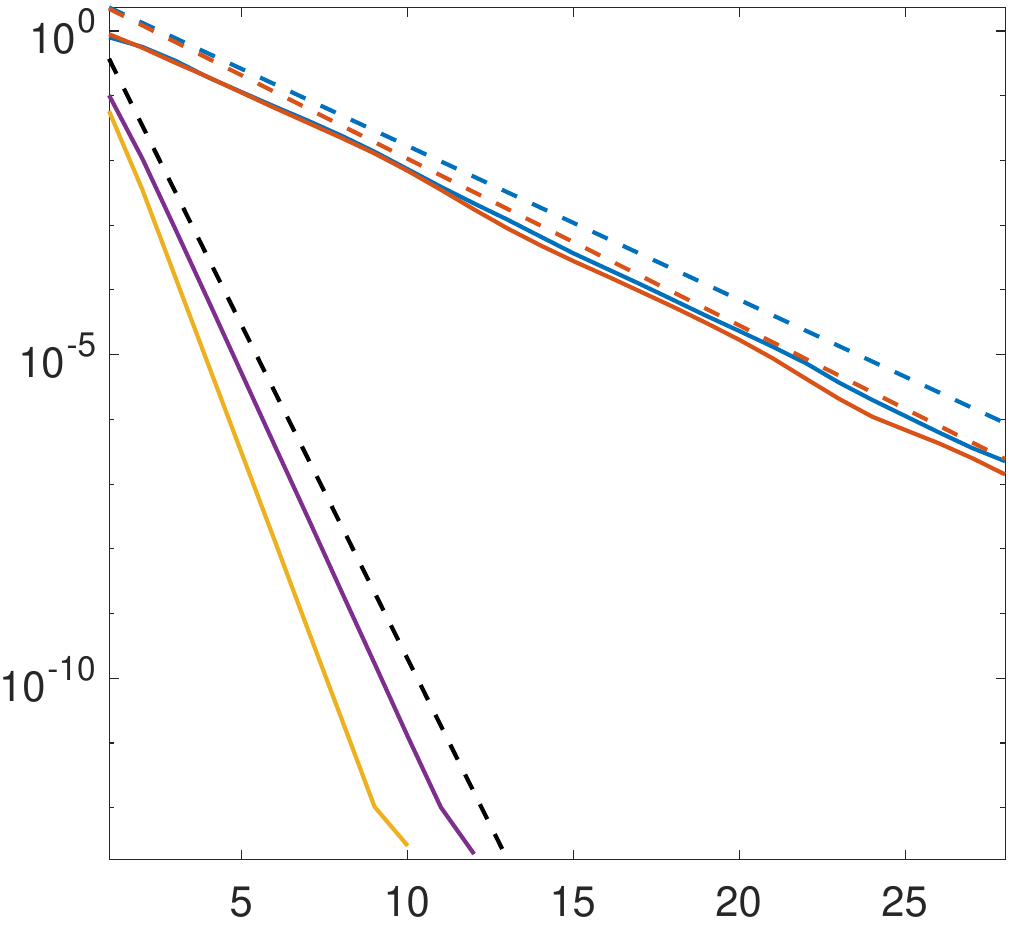} 
  \put(46,-4) {iterations}
  \put(-5, 40){\rotatebox{90}{ rel. error}}
  \put(48, 76){\rotatebox{-25}{$\sep = 1$}}
    \put(25, 62){\rotatebox{-62}{$\sep = m\!+\!1$}}
  \end{overpic}
  \caption{The error $\|C_v-C_v^{(k)}\|_2/\|C_v\|_2$, where $C_v^{(k)}$ is constructed with $k$ iterations of fADI, is plotted (solid lines) on a logarithmic scale against $k$ for various choices of submatrices $C_v$. Here, $C$ is a $4096 \times 4096$ Cauchy-like matrix. Dotted lines show the error bounds via \Cref{corr:ADIbounds}. The submatrices $C_v = C(1\!:\!2048 \, , \, 2049\!:\!4096)$ (blue) and $C_v = C(1\!:\!1024 \, , \, 1025\!:\!4096)$ (red) are $(m, \sep = 1)$ submatrices. In contrast,  matrices $C_v = C(1\!:\!1024 \, , \, 2049\!:\!3072)$ (yellow) and $C_v = C(1025\!:\!3584 \, , \, 1\!:\!512)$ (purple) are submatrices of type $(m, \sep = m\!+\!1)$. The relative error bound is the same for these latter two submatrices.}
  \label{fig:fadi_perf}
  \end{figure}

We now turn to a practical algorithm for constructing low rank approximations to off-diagonal HODLR blocks of $C$, by combining the construction of the rational approximation from Section~\ref{sec:poleszeros} with ADI.

By~\eqref{eq:dispY}, every off-diagonal HODLR block \smash{$X = C_v$} of $C$ satisfies a Sylvester matrix equation with diagonal coefficients:
\begin{equation} 
\label{eq:displacement_block}
D_JX - XD_K = \widetilde{G}_J\widetilde{H}_K^*, \quad J = J_v, \quad K = J_{\widetilde v}.
\end{equation} 
ADI is an iterative method that solves for $X$ by alternately updating the column and row spaces of an approximate solution~\cite{lu1991solution, peaceman1955numerical}.
The factored version of ADI (termed fADI)~\cite{benner2009adi,li2002low} is a mathematically equivalent reformulation of ADI that returns the approximation in factored form. After $k$ iterations, it produces factors $Z,W$ such that $X \approx ZW^*$. As $Z$ and $W$ each have $\rho k$ columns, it follows that the rank of $ZW^*$ is at most $\rho k$. 
 
\cref{alg:fadicauchy} provides the pseudocode of fADI applied to~\eqref{eq:displacement_block}. Note that each 
iteration of fADI is extremely cheap because $D_J$ and $D_K$ are diagonal. Assuming both diagonal matrices are $m\times m$, only $\mathcal{O}(\rho m)$ operations are needed per iteration.

The convergence speed of fADI critically depends on a suitable choice of \textit{shift parameters} $\{\tau_{j}, \nu_{j}\}_{j = 1}^k$ in~\cref{alg:fadicauchy}. As we will see in the next section, the results from Section~\ref{sec:poleszeros} make it possible to choose nearly optimal shift parameters so that 
$\mathcal{O}( \log m \log \epsilon^{-1} )$ iterations and, thus, $\mathcal{O}( \rho m \log m \log \epsilon^{-1} )$ operations suffice to construct a low rank approximation $ZW^*$ such that \hbox{$\|X - ZW^*\|_2/\|X\|_2 \leq \epsilon$} holds.

\subsubsection{Choice of shift parameters}
\begin{algorithm}[t!]
  \caption{fADI for the low rank approximation of a Cauchy-like matrix} \label{alg:fadicauchy}
  \begin{algorithmic}[1]
    \Procedure{\rm fADI}{$ D_{J}, D_{K}, \widetilde{G}_J, \widetilde{H}_K, \{\tau_j, \nu_j\}_{j = 1}^k$} 
\State $Z_1 = (\nu_1-\tau_1) (D_J - \nu_1 I)^{-1}\widetilde{G}_J$
\State$W_1 = (D_K^*-\overline{\tau}_1 I)^{-1}\widetilde{H}_K$
    \For{$ j = 1$ to $k\!-\!1$}
	\State $Z_{j +1} = (\nu_{j+1}-\tau_{j+1}) (D_J - \tau_jI) (D_J-\nu_{j+1} I)^{-1} Z_j$
	\State $W_{j+1} = (D_K^* - \overline{\nu}_j I) (D_{K}^* - \overline{\tau}_{j+1} I)^{-1} W_j$
    \EndFor
\State $Z = \left[ Z_1, Z_2, \cdots, Z_k \right]$
\State $W = \left[ W_1, W_2, \cdots, W_k \right]$
   \EndProcedure
\State {\bf Output:} $Z,W \in \mathbb{C}^{m \times \rho k}$ such that $X \approx ZW^*$, where $X$ is as in~\eqref{eq:displacement_block}. 
 \end{algorithmic}
\end{algorithm}
Letting \smash{$X^{(k)} = ZW^*$} denote the fADI-generated solution after $k$ iterations, the shift parameters should be chosen to minimize the error \smash{$\|X - X^{(k)}\|_2$}. This optimization problem has been studied extensively~\cite{Beckermann2019, lebedev1977zolotarev, lu1991solution, sabino2007solution}, and it is known that 
\begin{equation} 
\label{eq:ADIerror}
X - X^{(k)} = r_k(D_J) X r_k(D_K)^{-1},  \quad r_k(z) = \prod_{j = 1}^k \dfrac{z - \tau_j}{z-\nu_j}.
\end{equation} 
When choosing $r_k$ as the Zolotarev rational function for 
$Z_k(\mathcal{A}_J, \mathcal{A}_K)$ with the (smallest) arcs $\mathcal{A}_J, \mathcal{A}_K$ such that
$\Lambda(D_J) \subset \mathcal A_J $, $\Lambda(D_K) \subset \mathcal A_K$, it follows that
\begin{equation} \label{eq:bnd_errfadi}
\|X-X^{(k)}\|_2 \leq Z_k(\mathcal{A}_J, \mathcal{A}_K) \|X\|_2.
\end{equation}
Upper bounds on $Z_k(\mathcal{A}_J, \mathcal{A}_K)$ thus bound the fADI error after $k$ iterations. 

\begin{corollary}
\label{corr:ADIbounds}
Let $X$ be an $m\times m$ off-diagonal HODLR block of $C$, satisfying~\eqref{eq:displacement_block} for a right-hand side of rank at most $\rho$.
Consider the matrix \smash{$X^{(k)} = ZW^*$} constructed after $k$ iterations of fADI with the 
shift parameters \smash{$\{(\tau_j, \nu_j)\}_{j = 1}^k$} from \cref{lemma:Shift_parameters}.
Then $X^{(k)}$ has rank at most $\rho k$ and
$$ \|X-X^{(k)}\|_2 \leq 4 \xi^{-k} \|X\|_2,$$
with $0 < \xi <1$ defined in~\eqref{eq:bound_zolo_arcs} for $\sep = 1$.
\end{corollary} 
\begin{proof}
The result follows directly from combining~\eqref{eq:bnd_errfadi} with the bound~\eqref{eq:bound_zolo_arcs}. 
\end{proof} 

For each off-diagonal block $C_v$ with $v\in \mathcal T$, we can use \cref{corr:ADIbounds} to adaptively construct a low-rank  approximation $\widetilde{C}_v$  that satisfies a given error tolerance \smash{$\epsilon_v$}. Letting $m$ denote the number of rows in  $C_v$, we apply
\begin{equation} \label{eq:kv}
k_v =\big\lceil 2\pi^{-2} \log( 4 m ) \log\big( 4 \epsilon^{-1}_v \big) \big\rceil,
\end{equation}
 steps of fADI, using the shift parameters from \cref{lemma:Shift_parameters}, to construct this approximation. \Cref{fig:fadi_perf} illustrates the performance of fADI in practice using various submatrices of a $4096 \times 4096$ Cauchy-like matrix associated with a Toeplitz matrix with entries randomly selected from a uniform distribution over $[0, 1]$; there is a clear difference between cases where $\sep = 1$ and $\sep = m +1$.
 
\subsection{Constructing HODLR approximations to $C$}

By applying fADI to each individual off-diagonal HODLR block $C_v$
and computing the diagonal part of $C$ separately as explained in Section~\ref{sec:displacement}, one obtains 
a HODLR approximation $\widetilde{C} \approx C$.
\begin{corollary} 
Given $\epsilon > 0$, the construction described above requires 
\begin{equation} \label{eq:complexity}
 \mathcal O\big( \rho n \log^2\!n \cdot (\log \log n + |\log \epsilon| ) \big)
\end{equation}
operations to produce a HODLR matrix $\widetilde{C}$ 
such that $\| C- \widetilde{C}\|_2 \le
\epsilon \| C\|_2$, provided that the minimum block size of the HODLR partitioning remains bounded. The HODLR rank of 
$\widetilde{C}$ is bounded by $k\rho$ with $k = \mathcal O\big( \log n \cdot (\log \log n + |\log \epsilon| ) \big)$.

\end{corollary}
\begin{proof} The choice~\eqref{eq:kv} of $k_v$ ensures that $\|C_v - \widetilde{C}_v\|_2 \le \epsilon_v \|C_v\|_2 \le \epsilon_v \|C\|_2$. Set 
$
\epsilon_{v} = \epsilon / \log_2 n
$
for every $v \in \mathcal T$.
Because of the structure of the HODLR matrix, the $2$-norm of all errors committed on any fixed level $\ell$ remains bounded by $\epsilon \|C\|_2 / \log_2 n $. As there are at most $\log_2 n$ levels, the claimed error bound 
$\| C- \widetilde{C}\|_2 \le
\epsilon \| C\|_2$ follows from the triangle inequality.

For each of the $2^{\ell+1}$ off-diagonal HODLR blocks on level $\ell \ge 1$, fADI requires $\mathcal O(\rho n/2^\ell k_v)$ operations with $k_v = \mathcal O( \log(n/2^\ell) \log(\log n / \epsilon) )$ by~\eqref{eq:kv}. Summing up over all blocks and over all $\mathcal O(\log n)$ levels, this implies~\eqref{eq:complexity},
taking into account that the cost for recovering the diagonal part of $C$ and assembling the remaining entries of the (small) diagonal blocks in the HODLR partitioning is negligible.
The claimed bound on the HODLR rank follows directly by setting $k = \max k_{v}$.
\end{proof}

Given $\widetilde{C} \approx C$, the corresponding approximation $\widetilde T \approx T$ of the 
Toeplitz-like matrix $T$ is given by $\widetilde T = F^* \widetilde{C} F$.
Because $F$ is unitary, this approximation also satisfies $\| T- \widetilde{T}\|_2 \le
\epsilon \| T\|_2$.
Note that $\widetilde T$ is never formed explicitly but only used implicitly, as in Algorithm~\ref{alg:generalsolver}.

\subsubsection{A faster HODLR approximation}

It turns out that the structure of $C$ can be exploited further such that only one fADI computation per tree level is required (related to ideas mentioned in~\cite[Sec. 4]{xia2012superfast} and~\cite{chandrasekaran2007superfast}). For this purpose, we introduce the ``base" matrix $\mathscr{C}$
with entries 
\[
\mathscr{C}_{jk} = \frac{\omega^j \omega^k}{\omega^{2j} - \omega^{2k}} = \frac{1}{2\mathrm{i} \sin( \pi (j-k)/n )}, \quad j \neq k,  
\]
and $\mathscr{C}_{jj} = 0$, where $0 \leq j, k \leq n\!-\!1$. 
Note that $\mathscr{C}$ satisfies a displacement equation of the form~\eqref{eq:cauchylike-disp}:
\begin{equation} \label{eq:cauchydisp}
  D \mathscr{C} - \mathscr{C} D = f f^{\top}- {\rm diag}(f^2), \quad f = \big[ 1, \omega^1, \ldots, \omega^{n-1} \big]^{\top}. 
\end{equation}
In turn, \cref{thm:bounds}, \cref{lemma:Shift_parameters}, and \cref{corr:ADIbounds} can all be applied to the off-diagonal HODLR blocks of $\mathscr{C}$. However, it is easy to check that $\mathscr{C}$ is also a Hermitian Toeplitz matrix.
This implies that the $2^{\ell+1}$ off-diagonal HODLR blocks on a fixed level $\ell$ are nearly identical: Considering the first off-diagonal block $\mathscr{C}_v$ on level $\ell$, the other level-$\ell$ blocks in the upper and lower triangular part are identical to $\mathscr{C}_v$ and ${\mathscr{C}_v^*}$, respectively. Hence, it suffices
to construct a low-rank approximation for the block $\mathscr{C}_{v}$ and use this approximation (or its Hermitian transpose) for all other blocks on the same level. This reduces the $\log^2\!n$ factor to $\log n$ in the complexity estimate~\eqref{eq:complexity}.

To relate $\mathscr{C}$ to the Cauchy-like matrix $C$, let us adjust the right-hand side of the displacement equation~\eqref{eq:cauchylike-disp} by setting $\widehat G := \mathrm{diag}(f)^{-1} \widetilde G$,
$\widehat H^* := \widetilde{H}^* \mathrm{diag}(f)^{-1}$. It follows that the off-diagonal parts of $C$
and $\mathscr{C} \circ \widehat{G}\widehat{H}^*$ are identical. In particular, arbitrary off-diagonal blocks $X = C_{JK}$, 
$\mathscr{X} = \mathscr{C}_{JK}$ satisfy the relation
\begin{equation} 
\label{eq:hadamard_off}
X = \mathscr{X} \circ \widehat{G}(J,:) \widehat{H}(K,:)^* = \sum_{j = 0}^{\rho-1} {\rm diag}(\widehat G(J,  j))\, \mathscr{X} \,{\rm diag}(\widehat H(K , j))^*,
\end{equation}
where we used Matlab notation to indicate submatrices of $\widehat{G}$, $\widehat{H}$. Using the fact that diagonal matrices permute, one easily verifies (see, e.g.,~\eqref{eq:ADIerror}) that the $k$th iterates $X^{(k)}$ and $\mathscr{X}^{(k)}$ of fADI for $X$ and $\mathscr{X}$, respectively, are related through
\begin{equation} \label{eq:hadamard_adi}
 X^{(k)} = \mathscr{X}^{(k)} \circ \widehat{G}(J,:) \widehat{H}(K,:)^* = \sum_{j = 0}^{\rho-1} {\rm diag}(\widehat G(J,  j))\, \mathscr{X}^{(k)} \,{\rm diag}(\widehat H(K , j))^*.
\end{equation}
This shows that the error estimate of Corollary~\ref{corr:ADIbounds} is maintained.

Because the computations of $\widehat G, \widehat H$, and the diagonal do not add significantly to the cost, we still gain the complexity factor $\log n$ mentioned above, both in terms of memory and cost.
This assumes that \smash{$X^{(k)}$} is represented in the factored form~\eqref{eq:hadamard_adi}, with the low-rank matrix \smash{$\mathscr{X}^{(k)}$ }obtained by carrying out one fADI computation per level.
This implicit representation is well suited when, e.g., carrying out matrix-vector products with the corresponding approximations \smash{$\widetilde C$} and \smash{$\widetilde T$}. 
If each low-rank off-diagonal HODLR block \smash{$X^{(k)}$} is actually evaluated, by multiplying the factors of \smash{$\mathscr{X}^{(k)}$ }with the diagonal matrices, this complexity advantage is lost. However, one can still expect this to perform
better than applying ADI separately for every off-diagonal block.

\section{HSS approximation of Cauchy-like matrix $C$} 
\label{sec:HSSintro}

While the HODLR format has the advantage of being relatively simple to implement, significant savings in Steps 2 and 3 of Algorithm~\ref{alg:generalsolver} are possible if  more complicated hierarchical low-rank formats, such as the HSS format, are used. 

Compared to HODLR, HSS features additional recursive relations in the off-diagonal low-rank factors across partition levels~\cite{bebendorf2008hierarchical,hackbusch2015hierarchical, martinsson2011fast}. 
Specifically, let \smash{$\widetilde{C}$}  be a HODLR matrix, and let every off-diagonal HODLR block \smash{$\widetilde{C}_v$}
be expressed as \smash{$\widetilde{C}_v = U_v B_{v} V_{\tilde{v}^*}$}, where \smash{$\tilde{v}$} is the sibling vertex of $v$ in the tree $\mathcal T$ determining the HODLR partitioning. For simplicity, we assume that each such block is of equal rank $p$ and, hence, \smash{$B_{v,\tilde v}$} is $p \times p$. 
For an HSS matrix \smash{$\widetilde{C}$}, the left factor \smash{$U_v$} for a parent vertex $v$ admits the expression
\begin{equation} 
\label{eq:HSS_facteveryorU}
U_v = \begin{bmatrix} U_{2v+1} & 0 \\ 0 & U_{2v+2} \end{bmatrix} R_v, \quad R_v \in \mathbb{C}^{2p \times p}.
\end{equation}
In other words, the columns of \smash{$U_v$} are contained in the subspace spanned by the direct sum of the left factors at the children vertices.
The transfer of information from the children to the parent vertex is encoded in the transfer matrix \smash{$R_v$}. If the sibling \smash{$\tilde v$} is a parent vertex, the right factor \smash{$V_{\tilde v}$} also satisfies a similar relation:
\begin{equation} 
\label{eq:HSS_factorV}
V_{\tilde{v}} = \begin{bmatrix} V_{2\tilde v+1} & 0 \\ 0 & V_{2\tilde v+2} \end{bmatrix} W_{\tilde{v}}, \quad W_{\tilde{v}} \in \mathbb{C}^{2p \times p}.
\end{equation}
Expanding the recursions~\eqref{eq:HSS_facteveryorU}--\eqref{eq:HSS_factorV},
the factorization \smash{$\widetilde C_1 = U_1B_{1,2} V_2^* $} for the example from Figure~\ref{fig:HODLRandTree} takes the form
\[
 \underbrace{\begin{bmatrix} 
U_{7} & & & \\
& U_{8} & & \\
& & U_{9} & \\
& & & U_{10}
\end{bmatrix}}_{4 \nmin \times 4 p} \underbrace{\begin{bmatrix} R_3 & \\ & R_4\end{bmatrix}}_{4p \times 2p} \underbrace{R_1}_{2p\times p} \underbrace{B_{1,2}}_{p \times p} W_2^* \begin{bmatrix} W_5^* & \\ & W_6^*\end{bmatrix} \begin{bmatrix} 
V_{11}^* & & & \\
& V_{12}^* & & \\
& & V_{13}^* & \\
& & & V_{14}^*
\end{bmatrix},
\]
where $\nmin$ denotes the minimum block size.

In general, the recursion explained above allows one to store only the small matrices \smash{$R_v$}, \smash{$W_v$}, and \smash{$B_{v,\tilde v}$} for every parent vertex $v\in \mathcal T$. On the lowest level \smash{$\widehat{\ell}$} of the tree $\mathcal T$,
one explicitly stores the $\nmin \times \nmin$ diagonal blocks as well as \smash{$U_v, B_{v,\tilde v},$ and $V_v$}  for each of the \smash{$2^{\widehat{\ell}}$} leaf vertices $v$. We refer to the representation in terms of these small matrices as the \textit{HSS format} of \smash{$\widetilde{C}$}. Storing \smash{$\widetilde{C}$} in the HSS format requires only $\mathcal{O}(n p)$ memory, which is about a factor $\log_2 n$ smaller compared to the HODLR format with HODLR rank $p$.

\subsection{Superfast HSS-based solvers}

\label{sec:advantagesHSS} 

The HSS format not only reduces the cost for storing but also the cost for operating with a matrix $\widetilde{C}$. For example, a matrix-vector product \smash{$\widetilde{C}y$} can be computed in $\mathcal{O}(n p)$ operations~\cite[Alg. 1]{martinsson2011fast}.
Crucially, the HSS format also makes the solution of a linear system \smash{$\widetilde{C}\widetilde{x} = \widetilde{b}$} especially efficient to compute. In the following, we only sketch how such fast solvers work and refer to~\cite{chandrasekaran2006fast,xia2012superfast} for more details. After computing a so-called ULV factorization~\cite{chandrasekaran2006fast} -- a recursive unitary transformation to triangular form that takes advantage of the nested hierarchical structure -- a fast merge-and-backward-substitution routine, followed by a forward-substitution scheme, can be used to solve the linear system. This ULV factorization and the subsequent solution only require \smash{$\mathcal{O}(p^2n)$} and \smash{$\mathcal{O}(p n)$} operations, respectively.  In~\cite{xia2012superfast},  a related ULV-like solver is introduced that relaxes the requirement of  using unitary transformations, saving about a factor of 2 over the standard ULV solver~\cite[Sec.~3.6]{xia2012superfast}. For this approach to work, every \smash{$U_v, V_v, R_v$,} and \smash{$W_v$} in the HSS format need to originate from interpolative decompositions.\footnote{In our experiments in \cref{sec:practicalsolver}, we use the interpolative decompositions described in Section~\ref{sec:int_decomp}, but then orthogonalize and apply the standard ULV solver.}

Our focus in this work is not the reproduction of the HSS solvers already developed in~\cite{chandrasekaran2007superfast, martinsson2011fast, xia2012superfast}. Instead, we develop an fADI-based interpolative decomposition for appropriate HSS blocks. This HSS approximation is cheap and deterministic, and it allows for any HSS solver to be used, even the specialized one in~\cite{xia2012superfast}.

\subsection{HSS rows and columns, HSS rank}
\label{sec:HSSrowscols}

Our construction of an HSS approximation to $C$ proceeds via compressing the so-called HSS rows and columns~\cite{benzi2016matrices, martinsson2011fast,xia2010fast}.

We call
\[
 C_v^{\mathrm{row}} = C(J_v, J_v^c), \quad J_v^c = \{ 0, \ldots, n-1\} \setminus J_v
\]
an HSS row. This selects the block row containing \smash{$C_v$}, but excludes the diagonal block in that row, as shown in Figure~\ref{fig:HSSrows_cols}.
Similarly, an HSS column is defined as \smash{$C_{v}^{\mathrm{col}} = C(J_{v}^c, J_{v})$}.  
As discussed in~\cite{xia2010fast}, the construction of the HSS matrix \smash{$\widetilde{C}$} with each block \smash{$\widetilde{C}_v$} of rank at most $p$ is possible if and only if every HSS row and column of \smash{$\widetilde{C}$} has rank at most $p$. 
This suggests to define the \smash{$(\epsilon, \mathcal{T})$}-HSS rank of $C$  as the smallest $p$ such that \smash{$\erank(C^{\mathrm{row}}_v), \erank(C^{\mathrm{col}}_v) \leq p$} for every $v \in \mathcal{T} \setminus \{0\}$. Applying Theorem~\ref{thm:bounds}
to $( C^{\mathrm{row}}_v )^*$ and
$C^{\mathrm{col}}_v$ with $m \le n/2$ and $\sep = 1$, 
yields the following result.
\begin{lemma} 
\label{lemma:HSS_rank}
For a perfectly balanced binary
tree $\mathcal{T}$, the integer
\[
p = \rho \big\lceil 2\pi^{-2} \log (2n) \log(4 \epsilon^{-1} ) \big\rceil
\]
 is an upper bound on the \smash{$(\epsilon, \mathcal{T})$}-HSS rank of $C$.
\end{lemma}

\begin{figure}
\centering
\begin{tikzpicture}[scale=0.27, every node/.style={scale=0.87}]
\draw[black, thick](0,0)--(8,0);
\draw[black, thick] (0,0)--(0,8);
\draw[black, thick] (8,0)--(8,8);
\draw[black, thick] (0,8)--(8,8);
\draw[black, thick](4,0)--(4,8);
\draw[black, thick](0,4)--(8,4);
\draw[black, thick](2, 4)--(2, 8); 
\draw[black, thick](0, 6)--(4, 6); 
\draw[black, thick](6, 0)--(6, 4); 
\draw[black, thick](4, 2)--(8,2);
\draw[black, thick](0, 7)--(2, 7); 
\draw[black, thick](1,6)--(1,8);
\draw[black, thick](2, 5)--(4,5); 
\draw[black, thick](3, 4)--(3, 6); 
\draw[black, thick](4, 3)--(6,3); 
\draw[black, thick](5,2)--(5, 4); 
\draw[black, thick](6,1)--(8,1); 
\draw[black, thick](7, 0)--(7, 2); 
\draw [pattern=north west lines, pattern color=blue] (0,5) rectangle (2,6);
\draw [pattern=north west lines, pattern color=blue] (3,5) rectangle (8,6);
\draw [pattern=north west lines, pattern color=red] (0,4) rectangle (3,5);
\draw [pattern=north west lines, pattern color=red] (4,4) rectangle (8,5);
\node[] at (4, -1) {(a)}; 
\end{tikzpicture}
\begin{tikzpicture}[scale=0.27, every node/.style={scale=0.87}]
\hspace{10pt}
\draw[black, thick](0,0)--(8,0);
\draw[black, thick] (0,0)--(0,8);
\draw[black, thick] (8,0)--(8,8);
\draw[black, thick] (0,8)--(8,8);
\draw[black, thick](4,0)--(4,8);
\draw[black, thick](0,4)--(8,4);
\draw[black, thick](2, 4)--(2, 8); 
\draw[black, thick](0, 6)--(4, 6); 
\draw[black, thick](6, 0)--(6, 4); 
\draw[black, thick](4, 2)--(8,2);
\draw[black, thick](0, 7)--(2, 7); 
\draw[black, thick](1,6)--(1,8);
\draw[black, thick](2, 5)--(4,5); 
\draw[black, thick](3, 4)--(3, 6); 
\draw[black, thick](4, 3)--(6,3); 
\draw[black, thick](5,2)--(5, 4); 
\draw[black, thick](6,1)--(8,1); 
\draw[black, thick](7, 0)--(7, 2); 
\draw [pattern=north west lines, pattern color=purple] (0,4) rectangle (2,6);
\draw [pattern=north west lines, pattern color=purple] (4,4) rectangle (8,6);
\node[] at (4, -1) {(b)}; 
\end{tikzpicture}
\begin{tikzpicture}[scale=0.27, every node/.style={scale=0.87}]
\hspace{20pt}
\draw[black, thick](0,0)--(8,0);
\draw[black, thick] (0,0)--(0,8);
\draw[black, thick] (8,0)--(8,8);
\draw[black, thick] (0,8)--(8,8);
\draw[black, thick](4,0)--(4,8);
\draw[black, thick](0,4)--(8,4);
\draw[black, thick](2, 4)--(2, 8); 
\draw[black, thick](0, 6)--(4, 6); 
\draw[black, thick](6, 0)--(6, 4); 
\draw[black, thick](4, 2)--(8,2);
\draw[black, thick](0, 7)--(2, 7); 
\draw[black, thick](1,6)--(1,8);
\draw[black, thick](2, 5)--(4,5); 
\draw[black, thick](3, 4)--(3, 6); 
\draw[black, thick](4, 3)--(6,3); 
\draw[black, thick](5,2)--(5, 4); 
\draw[black, thick](6,1)--(8,1); 
\draw[black, thick](7, 0)--(7, 2); 
\draw [pattern=north west lines, pattern color=blue] (4,0) rectangle (6,2);
\draw [pattern=north west lines, pattern color=blue] (4,4) rectangle (6,8);
\draw [pattern=north west lines, pattern color=red] (6,2) rectangle (8,8);
\node[] at (4, -1) {(c)}; 
\end{tikzpicture}
\begin{tikzpicture}[scale=0.27, every node/.style={scale=0.87}]
\hspace{30pt}
\draw[black, thick](0,0)--(8,0);
\draw[black, thick] (0,0)--(0,8);
\draw[black, thick] (8,0)--(8,8);
\draw[black, thick] (0,8)--(8,8);
\draw[black, thick](4,0)--(4,8);
\draw[black, thick](0,4)--(8,4);
\draw[black, thick](2, 4)--(2, 8); 
\draw[black, thick](0, 6)--(4, 6); 
\draw[black, thick](6, 0)--(6, 4); 
\draw[black, thick](4, 2)--(8,2);
\draw[black, thick](0, 7)--(2, 7); 
\draw[black, thick](1,6)--(1,8);
\draw[black, thick](2, 5)--(4,5); 
\draw[black, thick](3, 4)--(3, 6); 
\draw[black, thick](4, 3)--(6,3); 
\draw[black, thick](5,2)--(5, 4); 
\draw[black, thick](6,1)--(8,1); 
\draw[black, thick](7, 0)--(7, 2); 
\draw [pattern=north west lines, pattern color=purple] (4,4) rectangle (8,8);
\node[] at (4, -1) {(d)}; 
\end{tikzpicture}
\caption{Various HSS rows and columns for a tree of depth $3$ (see Figure~\ref{fig:HODLRandTree}). In (a), the two HSS rows $C_9^{\mathrm{row}}$(blue) and $C_{10}^{\mathrm{row}}$ (red) are shown.  These are the children of the row $C_4^{\mathrm{row}}$, shown in (b). The two HSS columns $C_5^{\mathrm{col}}$ (red)  and $C_6^{\mathrm{col}}$ (blue)  in (c) are the children of the parent $C_2^{\mathrm{col}}$ in (d). }
\label{fig:HSSrows_cols}
\end{figure}

\subsection{Approximation by interpolative decompositions} 
\label{sec:int_decomp}

By~\cite[Thm. 4.7]{KMR2019}, there is an HSS matrix
$\widetilde C$ of HSS rank $p$, with the value of $p$ stated in Lemma~\ref{lemma:HSS_rank}, such that
$\|C- \widetilde C\|_2 \lesssim \sqrt{n} \epsilon$. Standard algorithms~\cite{xia2010fast} for constructing such a compression $\widetilde C$ proceed by recursively compressing the HSS blocks and rows in a bottom-up fashion. The compression on the upper levels by orthogonal projections is expensive, requiring already $\mathcal O(n^2)$ operations for just evaluating the entries of the corresponding HSS block columns/rows. Interpolative decompositions~\cite{cheng2005compression,martinsson2011fast, xia2012superfast} are a common approach to mitigate this problem. In this section, we explain how interpolative decompositions can be combined with fADI to cheaply construct HSS approximations to $C$.

\subsubsection{Interpolative decompositions on the leaf level}  \label{sec:leafinterpolation}

Consider a leaf $v \in \mathcal T$ with sibling $\tilde v$. Let \smash{$\mathbb J_v \subset J_v$} and \smash{$\mathbb K_{\tilde{v}} \subset J_{\tilde{v}}$} denote subsets, each of cardinality $p \ll m$. We aim at constructing 
a two-sided interpolative decomposition~\cite{cheng2005compression} for the corresponding off-diagonal block $C_v \in \C^{m \times m}$, $m = \nmin$, of the form
\begin{equation}
\label{eq:block_factors}
 C_v \approx \widetilde{C}_v = U_v B_{v, \tilde{v}} V_{\tilde{v}}^*, \quad B_{v, \tilde{v}} = C(\mathbb{J}_v, \mathbb{K}_{\tilde{v}}),
 \end{equation} 
 This decomposition is obtained from constructing one-sided 
interpolative decompositions for the corresponding HSS block row/column:
\begin{equation} 
\label{eq:HSSrowinterp2} 
C_v^{\mathrm{row}} \approx \widetilde{C}_v^{\mathrm{row}}  = U_v  C_v^{\mathrm{row}}(\mathbb{J}_v, \, :\,), \quad 
C_{\tilde v}^{\mathrm{col}} \approx \widetilde{C}_{\tilde v}^{\mathrm{col}}  = C_{\tilde v}^{\mathrm{col}}(\, :\,, \mathbb K_{\tilde v}) V_{\tilde v}^*.
\end{equation}
This requires that the selected subset $\mathbb{J}_v$ of rows represents a good approximation 
to the row space of $C_v^{\mathrm{row}}$
and, analogously, that the 
selected subset $\mathbb{K}_{\tilde v}$ of columns represents a good approximation 
to column space of $C_{\tilde v}^{\mathrm{col}}$. Both are instances of the column subset selection problem~\cite{broadbent2010subset}, for which numerous algorithms have been developed, including 
the column-pivoted QR (CPQR)~\cite{GolubVanLoan2013}, (strong) rank-revealing QR decomposition~\cite{gu1996efficient} or Osinsky's algorithm~\cite{Osinsky2023}.
However, applying such algorithms directly to HSS block rows/columns is expensive. For example, \smash{$C_v^{\mathrm{col}}$} is of size \hbox{$(n\!-\!m) \times m$}, so the cost depends at least linearly on the ``long" dimension $n$.  One can do better with carefully implemented randomized methods~\cite{martinsson2011fast,xia2012superfast}, but even this approach requires an initial precomputation that multiplies the uncompressed matrix $C$ with an $n \times \Theta(p)$ Gaussian random matrix.\footnote{In~\cite{xia2012superfast}, this step is greatly improved with a fast matrix-vector multiplication routine for $C$ based on its relationship to the Toeplitz matrix $T$.} 
In the following, we show how the interpolative decompositions~\eqref{eq:HSSrowinterp2} 
can be found deterministically using fADI with a cost that depends only on $m$ and $p$.

We focus on the selection of $\mathbb{J}_v$; the process for finding $\mathbb{K}_{\tilde v}$ is entirely analogous.
Set \smash{$X := C_v^{\mathrm{row}}$}.
We know from~\eqref{eq:cauchylike-disp} that $X$ satisfies \[
D_JX - XD_K = \widetilde G_J \widetilde H_K^*, \quad J = J_v, \quad K = J_v^{c}.
\]
After $k$ iterations of fADI, the approximation \begin{equation}\label{eq:adi2}
X \approx X^{(k)} = ZW^*, \quad  Z  \in \mathbb{C}^{m \times p }, \quad p = k \rho,
                                                \end{equation}
is constructed. To find an approximate one-sided row interpolative decomposition, the (large) matrix $W$ is actually not needed. A key advantage of fADI is that it constructs $Z$ independently from $W$ using only the $m \times \rho$ matrix \smash{$\widetilde G_J$} and the $m \times m$ diagonal matrix \smash{$D_J$} (see \cref{alg:fadicauchy}). In turn, the construction of $Z$ only requires $\mathcal{O}(m \rho)$ operations.

One continues from~\eqref{eq:adi2} by selecting a good subset of rows in the (smaller) factor $Z$. For convenience, 
our numerical experiments apply CPQR to $Z^*$ for this purpose. This has a computational cost of only \smash{$\mathcal{O}(m p^2)$} operations and is usually a safe choice.\footnote{Despite famous adversarial examples where CPQR fails, it is considered highly reliable in practice~\cite{gu1996efficient}.} Better theoretical bounds are obtained, at the same asymptotic cost, by using more sophisticated methods, such as Osinsky's algorithm~\cite{Osinsky2023}; see also~\cite{Cortinovis2024}. For this purpose, we first compute a QR decomposition $Z = QR$ with $Q \in \C^{m\times p}$. Applying Algorithm 2 from~\cite{Osinsky2023} to $Q^*$ returns an index set
$\mathbb J_v$ of cardinality $p$ such that
\begin{equation} \label{eq:osinsky}
 \|Z \cdot Z(\mathbb J_v,:)^{-1} \|_2 = 
 \|Q \cdot Q(\mathbb J_v,:)^{-1} \|_2 \le \sqrt{1+p(m-p)};
\end{equation}
see~\cite[Thm. 5]{Osinsky2023}. Setting $U_v = Z \cdot Z(\mathbb{J}_v,:)^{-1}$ concludes the construction of the desired
approximation $C_v^{\mathrm{row}} \approx U_v  C_v^{\mathrm{row}}(\mathbb{J}_v, \, :\,)$.

\begin{lemma} 
\label{lemma:int_decomp_error_leaf}
Set $X = C_v^{\mathrm{row}} \in \C^{m \times (n-m)}$.
The approximation $X \approx U_v X(\mathbb{J}_v, \, : \,)$ described above satisfies the error bound
\begin{equation} 
\label{eq:HSS_fadi_error}
\|X - U_v X(\mathbb{J}_v, \, : \,)\|_2 \leq 4 \xi^{-k} \big( 1 + \sqrt{1+p(m-p)}\big) \|X\|_2,
\end{equation}
with \smash{$\xi = \exp(  \pi^2 / ( 2 \log( 4 m  ) )$} and $p = k \rho$.
\end{lemma}
\begin{proof}
Observe that
$$ \|X - U_v X(\mathbb{J}_v, \, : \,)\|_2 \leq  \|X - ZW^*\|_2 + \|ZW^* - U_v X(\mathbb{J}_v, \, : \,)\|_2.$$
The second term is bounded by
\[
 \|ZW^* - U_v X(\mathbb{J}_v, \, : \,)\|_2
 = \big\|
 U_v \big(
 Z(\mathbb J_v,:) W^* - X(\mathbb{J}_v, \, : \,)
 \big) 
 \big\|_2 \le \|U_v\|_2\|  ZW^* - X  \|_2.
\]
The proof is concluded by using~\eqref{eq:osinsky} for bounding $\|U_v\|_2$, and $\|X - ZW^*\|_2 \leq 4 \xi^{-k} \|X\|_2$,
which follows from a straightforward extension of~\cref{corr:ADIbounds}.
\end{proof}

We view~\eqref{eq:HSS_fadi_error} as the deterministic analogue to the error bounds supplied in~\cite[Eq.~3.6]{xia2012superfast} for interpolative decompositions based on randomized linear algebra. There, one finds  \smash{$ \Theta_v X(\mathbb J_v, \; : \;)$}, which is a rank $p + \mu$ approximate row interpolative decomposition for $X$, with $\mu$ selected as a small  oversampling parameter~\cite{halko2011finding}. Properties of Gaussian random matrices can be used to show that with probability \smash{$1 - 6 \mu^{-\mu}$}, the error is bounded as~\cite[Eq.~(3.6)]{xia2012superfast} 
\begin{equation} 
\label{eq:Rand_err}
\|X - \Theta_v X(\mathbb J_v, \; : \;)\|_2 \leq \big( 1 + 11 \sqrt{m(p + \mu)}\big) \sigma_{p+1}(X). 
\end{equation}
Using~\eqref{eq:bound_zolo_arcs} and~\eqref{eq:zolobound_svs}, our bounds show that \smash{$\sigma_{p+1}(X) \leq  4 \xi^{-\lfloor p/\rho \rfloor } \| X\|_2$}, where $\xi$ is given in~\eqref{eq:bound_zolo_arcs}. This makes~\eqref{eq:Rand_err} explicit so that if one prefers to use a solver based on randomized linear algebra, $p$ can be chosen a priori in order to ensure (w.h.p.) that the relative error is bounded by $\epsilon$.

\subsection{HSS approximation by recursive interpolative decomposition}
\label{sec:merge}

Once the interpolative decomposition~\eqref{eq:block_factors} is known for each block on the finest level \smash{$\widehat{\ell}$} of \smash{$\mathcal{T}$}, the remaining factors associated with the non-leaf vertices of \smash{$\mathcal{T}$} are determined
in a bottom-up fashion. We briefly describe the process of merging from the leaf level up to the next coarsest level, but refer to~\cite{wilber2021computing} for substantial details on such processes, which are admittedly technical but also standard and well-established.

Let $v$ be a vertex on level \smash{$\ell = \widehat{\ell}\! -\! 1$}. Then, as before, we aim at finding an approximation \smash{$C_v \approx \widetilde{C}_v = U_vB_{v, \tilde{v}}V_{\tilde{v}}^*$}. It is assumed that the outer factor  \smash{$U_v$} and \smash{$V_{\tilde{v}}$} satisfy the relations~\eqref{eq:HSS_facteveryorU} and~\eqref{eq:HSS_factorV}, respectively, and it thus only remains to determine \smash{$R_v$}, \smash{$W_{\tilde{v}}$}, and the index sets $\mathbb J_v$, $\mathbb K_{\tilde{v}}$ defining \smash{$B_{v, \tilde{v}} = C(\mathbb J_v, \mathbb K_{\tilde{v}})$}.  Consider the associated HSS row \smash{$C_v^{\mathrm{row}}$}. Its children HSS rows \smash{$C_{2v+1}^{\mathrm{row}}$, $C_{2v+2}^{\mathrm{row}}$} are approximated by the algorithm discussed in Section~\ref{sec:leafinterpolation}:
$$C_{2v+j}^{\mathrm{row}}(\, : \, , \, J_v^c) \approx U_{2 v+j} C_{2v+j}^{\mathrm{row}}(\mathbb{J}_{2v+j}, \, J_v^c\,).$$ Using these decompositions, we can write
\begin{equation} 
\label{eq:non_leaf_row}
 C_v^{\mathrm{row}} \approx \begin{bmatrix} U_{2v+1} & 0 \\ 0 & U_{2v+2} \end{bmatrix} C_v(\widehat{\mathbb J}, \, : \,),
 \end{equation} 
where \smash{$\widehat{\mathbb J}$} is the index set with respect to the block \smash{$C_v^{\mathrm{row}}$} that selects the $2p$ rows indexed by \smash{$\mathbb{J}_{2v+1} \cup \mathbb{J}_{2v+2}$}; see~\cref{fig:basis_rows}~(b).  

With the cascading factors in~\eqref{eq:HSS_facteveryorU} in mind, we simply apply the fADI-based row interpolative decomposition from Section~\ref{sec:leafinterpolation} to the submatrix $C_v(\widehat{\mathbb J},:)$.  Theorem~\ref{thm:bounds} and Lemma~\ref{lemma:int_decomp_error_leaf} hold for this submatrix, which is also Cauchy-like and therefore has a useful displacement structure. 
This process chooses a size-$p$ subset $\mathbb J_v$ of the $2 p$ rows comprising \smash{$C_v(\widehat{\mathbb{J}}, \, : \,)$}; see \cref{fig:basis_rows}~(c), and it also produces $R_v$.  An analogous method is applied to \smash{$C_{\tilde{v}}^{\mathrm{col}}$} to find \smash{$\mathbb K_{\tilde{v}}$} and \smash{$W_{\tilde{v}}$}. This procedure can be repeated at each level as we traverse \smash{$\mathcal{T}$} from the bottom to the top in order to produce an approximate HSS factorization of the entire matrix $C$~\cite{martinsson2011fast}. 

\begin{figure}
\centering
\label{fig:basis_rows}
\begin{tikzpicture}[scale=0.45, every node/.style={scale=0.7}]
\draw[black](0,0)--(0,2);
\draw[black] (0,0)--(8,0);
\draw[black] (0,2)--(8,2);
\draw[black] (8,0)--(8,2);
\draw[black] (2,0)--(2,1);
\draw[black] (0,1)--(8,1);
\draw[black] (1,0)--(1,2);
\draw[purple, thick](1, 1.5)--(8, 1.5); 
\draw[purple, thick](1, 1.2)--(8, 1.2); 
\draw[purple, thick](1, 1.8)--(8, 1.8); 
\draw[blue, thick](2, .4)--(8, .4); 
\draw[blue, thick](2, .23)--(8, .23); 
\draw[blue, thick](2, .7)--(8, .7);  
\draw[blue, thick](0, .4)--(1, .4); 
\draw[blue, thick](0, .23)--(1, .23); 
\draw[blue, thick](0, .7)--(1, .7);  
\node[] at (4, -1) {(a)}; 
\end{tikzpicture}
\begin{tikzpicture}[scale=0.45, every node/.style={scale=0.7}]

\hspace{8pt} 

\draw[black](0,0)--(0,2);
\draw[black] (0,0)--(8,0);
\draw[black] (0,2)--(8,2);
\draw[black] (8,0)--(8,2);
\draw[black] (2,0)--(2,2);
\draw[black] (0,1)--(2,1);
\draw[black] (1,0)--(1,2);
\draw[purple, thick](2, 1.5)--(8, 1.5); 
\draw[purple, thick](2, 1.2)--(8, 1.2); 
\draw[purple, thick](2, 1.8)--(8, 1.8); 
\draw[blue, thick](2, .4)--(8, .4); 
\draw[blue, thick](2, .23)--(8, .23); 
\draw[blue, thick](2, .7)--(8, .7);  
\node[] at (4, -1) {(b)}; 
\end{tikzpicture}
\begin{tikzpicture}[scale=0.45, every node/.style={scale=0.7}]

\hspace{16pt}

\draw[black](0,0)--(0,2);
\draw[black] (0,0)--(8,0);
\draw[black] (0,2)--(8,2);
\draw[black] (8,0)--(8,2);
\draw[black] (2,0)--(2,2);
\draw[black] (0,1)--(2,1);
\draw[black] (1,0)--(1,2);
\draw[purple, thick](2, 1.5)--(8, 1.5); 
\draw[purple, thick](2, 1.8)--(8, 1.8); 
\draw[blue, thick](2, .4)--(8, .4); 
\node[] at (4, -1) {(c)}; 
\end{tikzpicture}
\caption{ (a) Selected row vectors indexed by \smash{$\mathbb{J}_{2v+j}$} for the two HSS rows \smash{$C_{2v+j}^{\mathrm{row}} \approx U_{2v+j} C_{2v+j}^{\mathrm{row}}(\mathbb{J}_{2v+j}, \, : \,)$},  $j = 1$ (red) and $j = 2$ (blue).  (b) The HSS row \smash{$C_v^{\mathrm{row}}$} is over-resolved by~\eqref{eq:non_leaf_row}, a row interpolative decomposition with  basis vectors taken as the union of the subsampled basis vectors (colored lines) of the children rows.  At this coarser level, we re-apply  interpolative decomposition to the submatrix of \smash{$C_v^{\mathrm{row}}$} consisting only of the colored basis row vectors in (b). This compresses the representation  by choosing a subselection of the colored row vectors as a new, reduced basis approximately spanning the row space of \smash{$C_v^{\mathrm{row}}$}, as shown in (c).}
\end{figure}

\subsection{A faster HSS factorization scheme}
\label{sec:cauchyhadamard}
As with the HODLR approximation to $C$, the identity~\eqref{eq:cauchydisp} can be used to improve the overall cost of the HSS factorization.  We take advantage of this when performing fADI on the lowest level $\hat \ell$ of the partitioning determined by $\mathcal T$. It suffices to find fADI approximants only for the first HSS row and column of $\mathscr{C}$.  Then, we use these approximations and~\eqref{eq:hadamard_off} to cheaply compute the left factors of the low-rank approximations for every level-$\hat \ell$ HSS row and column of $C$. Once this is done, CPQR or Osinsky's algorithm are used find interpolative decompositions from each of these low rank approximations. 

In principle, one could go further with this idea, and cheaply compute interpolative decompositions \smash{$\mathscr{C}_v = \mathcal{U}_v \mathscr{C}(\mathbb{J}_v, \mathbb{K}_{\tilde{v}}) \mathcal{V}_{\tilde{v}}^*$} on every partition level. When the structure of $\mathscr{C}$ is fully exploited and approximations to the submatrices of $C$ are stored in the factored form shown as in~\eqref{eq:hadamard_adi}, the cost for finding an approximate HSS factorization only requires \smash{$\mathcal{O}(\rho^3 \log^3 n \log^3 1/\epsilon)$} operations, that is, this step only depends poly-logarithmically (and not even linearly) on $n$.  This is an effective format for some operations (e.g., matrix-vector products). For solving linear systems,~\eqref{eq:hadamard_off} must be used to assemble low rank approximations to each \smash{$C_v$}. A similar idea is discussed in~\cite{chandrasekaran2007superfast, xia2012superfast}. We note that applying~\eqref{eq:hadamard_off} destroys the interpolative structure of the HSS factors, ruling out the use of the accelerated ULV-like solver from~\cite{xia2012superfast} in Step 3 of Algorithm~\ref{alg:generalsolver}.

\section{A fast HSS-based solver} \label{sec:practicalsolver}

Without exploiting~\eqref{eq:hadamard_adi}, the computational complexity of constructing an HSS approximation to $C$ with our fADI-based interpolative decomposition scheme requires $\mathcal{O}(n (\rho \log n \log 1/\epsilon )^2 )$ operations; a more detailed complexity
analysis can be found in~\cite[Ch.~4]{wilber2021computing}. Putting this together with existing HSS-based solvers for linear systems, the overall complexity for solving $Tx = b$ via \cref{alg:generalsolver} is also $\mathcal{O}(n (\rho \log n \log 1/\epsilon )^2 )$. 

We remark that the bounds in \cref{thm:bounds} indicate that a solver for $Tx=b$ with linear asymptotic complexity in $n$ appears to be possible. This is because submatrices of $C$ sufficiently separated from the main diagonal have $\epsilon$-ranks that are constant with respect to $n$. An appropriate hierarchical matrix approximation with strong admissibility could exploit this fact. Linear inversion routines for such matrices are available~\cite{ambikasaran2014inverse, minden2017recursive}, but the analysis and implementation of these methods are beyond the scope of this paper.

\subsection{Numerical results}

\begin{figure}[t!]
 \centering
\includegraphics[scale = .41]{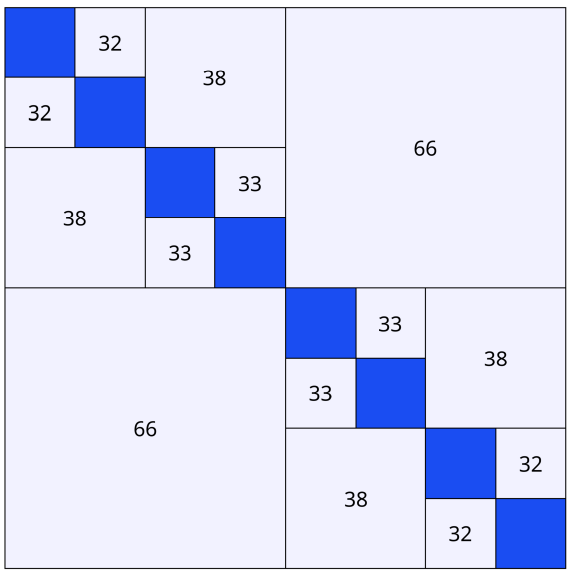}
\put(-112, 116){ \tiny $\|C-\widetilde C\|_2/\|C\|_2 \approx 1.8 \times 10^{-14}$}
\includegraphics[scale = .42]{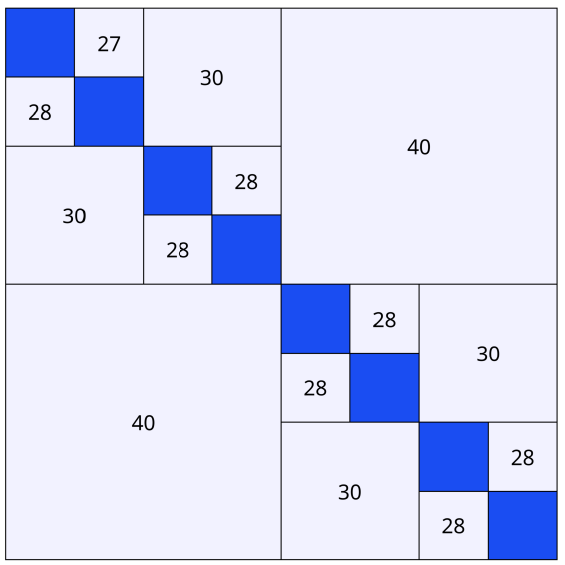}
\put(-108, 117){ \tiny $\|C-\widetilde C\|_2/\|C\|_2 \approx 6.8 \times 10^{-8}$}
\includegraphics[scale = .41]{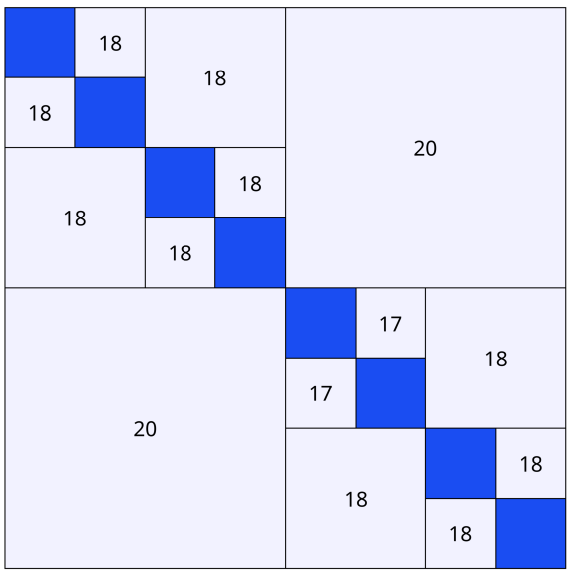}
\put(-108, 117){ \tiny $\|C-\widetilde C\|_2/\|C\|_2 \approx 2.0 \times 10^{-4}$}
  \caption{The ranks of off-diagonal blocks of an HSS approximation $\widetilde C\approx C$ with 4 levels for different choices of the tolerance $\epsilon$. The ranks of each block are adaptively determined using the bounds in \Cref{thm:bounds} and then low rank approximations are constructed using the fADI-based interpolative decomposition.}
  \label{fig:ranksHSS}
  \end{figure}

To validate the method based on our theoretical results, we include preliminary numerical results based on an implementation of Algorithm~\ref{alg:generalsolver} in MATLAB that uses the fADI-based interpolative decomposition scheme to construct an HSS approximation to $C$. Our method is built on top of the hm-toolbox~\cite{massei2020hm} that can be used to construct and compute with HSS, HODLR, and similar matrices. %

\begin{figure} 
 \begin{minipage}{.42\textwidth} 
 \centering
 \hspace{1cm}
  \begin{overpic}[width=\textwidth]{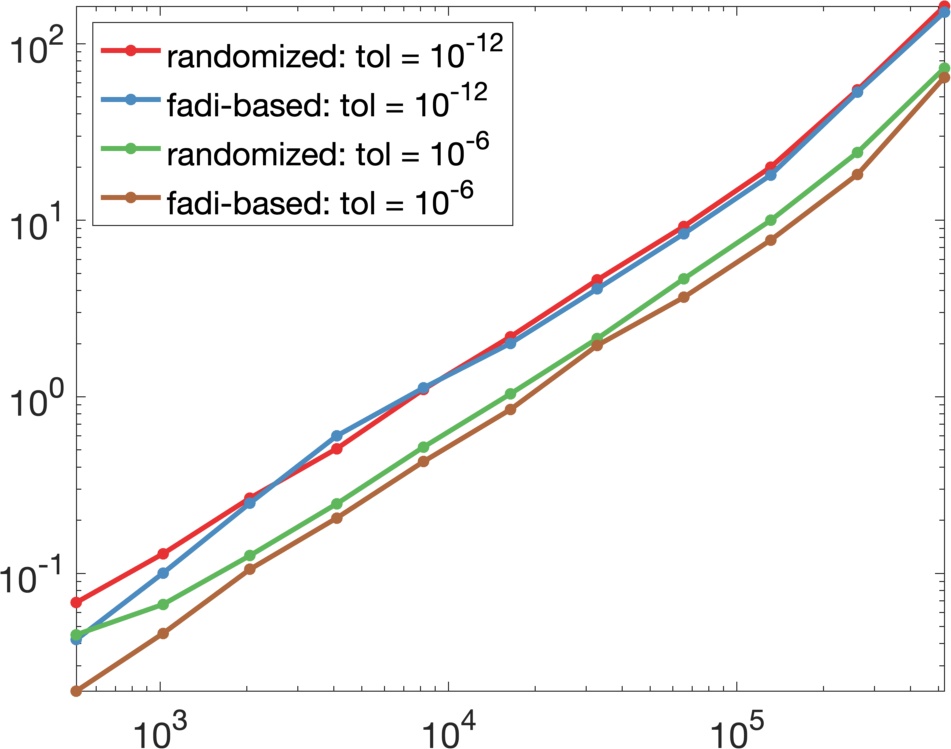}
  \put(45, -5){ \small k}
  \put(-7, 30){\rotatebox{90}{ \small rel. error}}
  \end{overpic}
  \end{minipage} \,
   \begin{minipage}{.4\textwidth} 
  \begin{tabular}{| c | c | c |}
 \hline
 $\epsilon$ & $\|C - \widetilde{C}\|_2/\|C\|_2$  & $\|x - \widehat{x}\|_2/\|x\|_2$    \\
 \hline  
 $10^{-3}$ & $1.887 \times 10^{-3}$ & $5.648 \times 10^{-3}$  \\
 $10^{-6}$ & $4.567 \times 10^{-7}$ & $9.110 \times 10^{-7}$  \\
  $10^{-9}$ & $3.623 \times 10^{-12}$ & $4.611 \times 10^{-11}$  \\
    $10^{-12}$ & $6.445 \times 10^{-14}$ & $3.431 \times 10^{-13}$  \\
 \hline
\end{tabular}
 \hspace{1cm}

  \end{minipage}
  \vspace{.2cm}
  \caption{Left: Time in seconds required for constructing the HSS approximation $\widetilde C$ of $C$ versus matrix size $n$. We compare the fADI-based scheme described in Section~\ref{sec:HSSintro} with the randomized method described in~\cite{xia2012superfast}. Tests were performed for two different tolerance parameter settings. Right: The ADI-based fast solver for $Tx =b$ comes with explicit low rank approximation error bounds and can be tuned to choose the $(\epsilon, \mathcal T)$-HSS rank adaptively to a given tolerance. The table displays the relative accuracy of $\widetilde{C}$, an approximate HSS factorization of $C$, and the relative accuracy of the computed solution $\widehat{x}$, for various choices of  $\epsilon$. For these tests, the Toeplitz matrices $T$ and right-hand sides $b$ were chosen randomly.}
  \label{fig:timingsHSS}
  \end{figure}
  
Our bounds on the low-rank approximation error can be used to adaptively choose the parameters given a tolerance $\epsilon > 0$.  Specifically, we use the bound of Theorem~\ref{thm:bounds} to determine $p$ so that \smash{$\erank(C_v^{\mathrm{row}})$ and  $\erank(C_v^{\mathrm{col}})$} are bounded by $p$ for every $v \in \mathcal{T}$. Along with these bounds, we automatically compute quasi-optimal fADI shift parameters via \cref{lemma:Shift_parameters} as needed for the fADI-based compression of the HSS rows and columns. This simplified strategy is designed so that \smash{$\|C -\widetilde{C}\|_2 \approx \epsilon \|C\|_2$}. The adaptivity of the ranks is shown for an HSS approximation to a $1024 \times 1024$ Cauchy-like matrix in~\Cref{fig:ranksHSS}. The method works well in practice, as demonstrated in Figure~\ref{fig:timingsHSS}.  Alternatively, one can devise a more elaborate scheme that uses Lemma~\ref{lemma:HSS_rank} to guarantee (in exact arithmetic) that \smash{$\|C - \widetilde{C} \|_2 \leq \epsilon \|C\|_2$}.  

As mentioned, we use CPQR within the construction of $\widetilde{C}$ in Step 2 of Algorithm~\ref{alg:generalsolver}.
Once the HSS matrix $\widetilde{C}$ is determined, we apply the built-in ULV solver in \texttt{hm-toolbox} to solve $\widetilde{C}\widetilde{x} = \widetilde{b}$ with only $\mathcal{O}(n p)$ flops. One could instead implement the ULV-like solver from~\cite{xia2012superfast} that omits orthogonalization to gain even more efficiency in the solve stage.

Our implementation takes  advantage of the similarity structure of $C$ on the leaf level, but not on the upper level.
 We observe that the performance of the method  matches that of the  Toeplitz solver in~\cite{xia2012superfast}, which is based on an HSS approximation of $C$ via randomized numerical linear algebra. To compare the methods, we implemented the solver in~\cite{xia2012superfast} in MATLAB, using our bounds to guide the estimates on the ranks of the HSS rows and columns.  Numerical results in Figure~\ref{fig:timingsHSS} confirm that these different compression schemes perform similarly in practice.  In this example, an $n \times n$ Toeplitz matrix is constructed by selecting entries for each diagonal randomly from a uniform distribution over $[0, 1]$. The results verify the asymptotic complexity of the method.

 \section{Conclusions}
 
 We have developed new ADI-based strategies for cheaply and reliably approximating Toeplitz-like matrices with hierarchical low-rank matrices, specifically HODLR and HSS matrices. We have demonstrated how such an approximation can be used to develop a competitive solver for linear systems with Toeplitz structure. Other potential applications include
 eigenvalue problems~\cite{Ou2022}, matrix functions~\cite{Casulli2024}, and matrix equations~\cite{KMR2019} involving Toeplitz and Toeplitz-like matrices.

\section*{Acknowledgments}

We are grateful to Robert Luce for early work and invaluable conversations on this topic. We also thank Alex Townsend and Anil Damle for reading and discussing an early version of the manuscript with us, and we thank an anonymous referee for many helpful suggestions.

\bibliographystyle{siam}
\bibliography{refs}

\begin{thebibliography}{10}

\bibitem{Akhiezer1990}
{\sc N.~I. Akhiezer}, {\em Elements of the theory of elliptic functions},
  vol.~79 of Translations of Mathematical Monographs, AMS, Providence, RI,
  1990.

\bibitem{ambikasaran2014inverse}
{\sc S.~Ambikasaran and E.~Darve}, {\em The inverse fast multipole method},
  arXiv preprint arXiv:1407.1572,  (2014).

\bibitem{antoulas2005approximation}
{\sc A.~C. Antoulas}, {\em Approximation of large-scale dynamical systems},
  vol.~6, SIAM, 2005.

\bibitem{benzi2016matrices}
{\sc J.~Ballani and D.~Kressner}, {\em Matrices with hierarchical low-rank
  structures}, vol.~2173 of Lecture Notes in Math., Springer, Cham, 2016,
  pp.~161--209.

\bibitem{bebendorf2008hierarchical}
{\sc M.~Bebendorf}, {\em Hierarchical matrices}, Springer, 2008.

\bibitem{Beckermann2019}
{\sc B.~Beckermann and A.~Townsend}, {\em Bounds on the singular values of
  matrices with displacement structure}, SIAM Rev., 61 (2019), pp.~319--344.

\bibitem{benner2009adi}
{\sc P.~Benner, R.-C. Li, and N.~Truhar}, {\em On the {ADI} method for
  {S}ylvester equations}, J. Comput. Appl. Math., 233 (2009), pp.~1035--1045.

\bibitem{Boerm2010}
{\sc S.~B\"{o}rm}, {\em Efficient numerical methods for non-local operators},
  vol.~14 of EMS Tracts in Mathematics, EMS, Z\"{u}rich, 2010.

\bibitem{braess2009efficient}
{\sc D.~Braess and W.~Hackbusch}, {\em On the efficient computation of
  high-dimensional integrals and the approximation by exponential sums}, in
  Multiscale, nonlinear and adaptive approximation, Springer, Berlin, 2009,
  pp.~39--74.

\bibitem{broadbent2010subset}
{\sc M.~E. Broadbent, M.~Brown, K.~Penner, I.~Ipsen, and R.~Rehman}, {\em
  Subset selection algorithms: Randomized vs. deterministic}, {SIAM}
  undergraduate research online, 3 (2010), pp.~50--71.

\bibitem{Casulli2024}
{\sc A.~A. Casulli, D.~Kressner, and L.~Robol}, {\em Computing functions of
  symmetric hierarchically semiseparable matrices}, SIAM J. Matrix Anal. Appl.,
  45 (2024), pp.~2314--2338.

\bibitem{chandrasekaran2006fast}
{\sc S.~Chandrasekaran, M.~Gu, and T.~Pals}, {\em A fast {ULV} decomposition
  solver for hierarchically semiseparable representations}, SIAM J. Matrix
  Anal. Appl., 28 (2006), pp.~603--622.

\bibitem{chandrasekaran2007superfast}
{\sc S.~Chandrasekaran, M.~Gu, X.~Sun, J.~Xia, and J.~Zhu}, {\em A superfast
  algorithm for {T}oeplitz systems of linear equations}, SIAM J. Matrix Anal.
  Appl., 29 (2007), pp.~1247--1266.

\bibitem{cheng2005compression}
{\sc H.~Cheng, Z.~Gimbutas, P.-G. Martinsson, and V.~Rokhlin}, {\em On the
  compression of low rank matrices}, SIAM J. Sci. Comput., 26 (2005),
  pp.~1389--1404.

\bibitem{Cortinovis2024}
{\sc A.~Cortinovis and D.~Kressner}, {\em Adaptive randomized pivoting for
  column subset selection, {DEIM}, and low-rank approximation}, arXiv preprint
  arXiv:2412.13992,  (2024).

\bibitem{driscoll1994schwarz}
{\sc T.~A. Driscoll}, {\em {S}chwarz-{C}hristoffel toolbox user's guide}, tech.
  rep., Cornell University, 1994.

\bibitem{druskin2011analysis}
{\sc V.~Druskin, L.~Knizhnerman, and V.~Simoncini}, {\em Analysis of the
  rational {K}rylov subspace and {ADI} methods for solving the {L}yapunov
  equation}, SIAM J. Numer. Anal., 49 (2011), pp.~1875--1898.

\bibitem{dutt1995fast}
{\sc A.~Dutt and V.~Rokhlin}, {\em Fast {F}ourier transforms for nonequispaced
  data, {II}}, Appl. Comput. Harm.Anal., 2 (1995), pp.~85--100.

\bibitem{Fortunato2020}
{\sc D.~Fortunato and A.~Townsend}, {\em Fast {P}oisson solvers for spectral
  methods}, IMA J. Numer. Anal., 40 (2020), pp.~1994--2018.

\bibitem{gohberg1995fast}
{\sc I.~Gohberg, T.~Kailath, and V.~Olshevsky}, {\em Fast {G}aussian
  elimination with partial pivoting for matrices with displacement structure},
  Math. of Comp., 64 (1995), pp.~1557--1576.

\bibitem{GolubVanLoan2013}
{\sc G.~H. Golub and C.~F. Van~Loan}, {\em Matrix computations}, Johns Hopkins
  Studies in the Mathematical Sciences, Johns Hopkins University Press,
  Baltimore, MD, fourth~ed., 2013.

\bibitem{Grasedyck2004}
{\sc L.~Grasedyck}, {\em Existence of a low rank or {$\mathcal{H}$}-matrix
  approximant to the solution of a {S}ylvester equation}, Numer. Linear Algebra
  Appl., 11 (2004), pp.~371--389.

\bibitem{gray2006toeplitz}
{\sc R.~M. Gray}, {\em Toeplitz and circulant matrices: A review}, Found. and
  Trends{\textregistered} in Comm. and Inf. Theory, 2 (2006), pp.~155--239.

\bibitem{grenander1958toeplitz}
{\sc U.~Grenander and G.~Szeg\"o}, {\em Toeplitz forms and their applications},
  California Monographs in Mathematical Sciences, University of California
  Press, Berkeley-Los Angeles, 1958.

\bibitem{gu1996efficient}
{\sc M.~Gu and S.~C. Eisenstat}, {\em Efficient algorithms for computing a
  strong rank-revealing {QR} factorization}, SIAM J. Sci. Comput., 17 (1996),
  pp.~848--869.

\bibitem{hackbusch2015hierarchical}
{\sc W.~Hackbusch}, {\em Hierarchical matrices: algorithms and analysis},
  vol.~49, Springer, 2015.

\bibitem{hackbusch2004hierarchical}
{\sc W.~Hackbusch, B.~N. Khoromskij, and R.~Kriemann}, {\em Hierarchical
  matrices based on a weak admissibility criterion}, Computing, 73 (2004),
  pp.~207--243.

\bibitem{halko2011finding}
{\sc N.~Halko, P.-G. Martinsson, and J.~A. Tropp}, {\em Finding structure with
  randomness: Probabilistic algorithms for constructing approximate matrix
  decompositions}, SIAM Rev., 53 (2011), pp.~217--288.

\bibitem{heinig2013algebraic}
{\sc G.~Heinig and K.~Rost}, {\em Algebraic methods for {T}oeplitz-like
  matrices and operators}, vol.~13 of Operator Theory: Advances and
  Applications, Birkh\"auser Verlag, Basel, 1984.

\bibitem{Huckle1998}
{\sc T.~Huckle}, {\em Superfast solution of linear equations with low
  displacement rank}, in High performance algorithms for structured matrix
  problems, vol.~2 of Adv. Theory Comput. Math., Nova Sci. Publ., Commack, NY,
  1998, pp.~149--162.

\bibitem{kailath1995displacement}
{\sc T.~Kailath and A.~H. Sayed}, {\em Displacement structure: theory and
  applications}, SIAM Rev., 37 (1995), pp.~297--386.

\bibitem{KMR2019}
{\sc D.~Kressner, S.~Massei, and L.~Robol}, {\em Low-rank updates and a
  divide-and-conquer method for linear matrix equations}, SIAM J. Sci. Comput.,
  41 (2019), pp.~A848--A876.

\bibitem{lebedev1977zolotarev}
{\sc V.~Lebedev}, {\em On a {Z}olotarev problem in the method of alternating
  directions}, USSR Comput. Math. Math.Phys., 17 (1977), pp.~58--76.

\bibitem{Lepilov2024}
{\sc M.~Lepilov and J.~Xia}, {\em Rank-structured approximation of some
  {C}auchy matrices with sublinear complexity}, Numer. Linear Algebra Appl., 31
  (2024), pp.~Paper No. e2526, 22.

\bibitem{li2002low}
{\sc J.-R. Li and J.~White}, {\em Low rank solution of {L}yapunov equations},
  SIAM J. Matrix Anal. Appl., 24 (2002), pp.~260--280.

\bibitem{lu1991solution}
{\sc A.~Lu and E.~L. Wachspress}, {\em Solution of {L}yapunov equations by
  alternating direction implicit iteration}, Comput. Math. Appl., 21 (1991),
  pp.~43--58.

\bibitem{martinsson2011fast}
{\sc P.-G. Martinsson}, {\em A fast randomized algorithm for computing a
  hierarchically semiseparable representation of a matrix}, SIAM J. Matrix
  Anal. Appl., 32 (2011), pp.~1251--1274.

\bibitem{martinsson2005fast}
{\sc P.-G. Martinsson, V.~Rokhlin, and M.~Tygert}, {\em A fast algorithm for
  the inversion of general {T}oeplitz matrices}, Comput. Math. Appl., 50
  (2005), pp.~741--752.

\bibitem{massei2020hm}
{\sc S.~Massei, L.~Robol, and D.~Kressner}, {\em hm-toolbox: {M}atlab software
  for {HODLR} and {HSS} matrices}, SIAM J. Sci. Comput., 42 (2020),
  pp.~C43--C68.

\bibitem{minden2017recursive}
{\sc V.~Minden, K.~L. Ho, A.~Damle, and L.~Ying}, {\em A recursive
  skeletonization factorization based on strong admissibility}, Multiscale
  Model. Simul., 15 (2017), pp.~768--796.

\bibitem{olver2010nist}
{\sc F.~W. Olver, D.~W. Lozier, R.~F. Boisvert, and C.~W. Clark}, {\em {NIST}
  handbook of mathematical functions}, Cambridge University Press, 2010.

\bibitem{Oseledets2007}
{\sc I.~V. Oseledets}, {\em Lower bounds for separable approximations of the
  {H}ilbert kernel}, Mat. Sb., 198 (2007), pp.~137--144.

\bibitem{Osinsky2023}
{\sc A.~Osinsky}, {\em Close to optimal column approximations with a single
  {SVD}}, arXiv preprint arXiv:2308.09068,  (2023).

\bibitem{Ou2022}
{\sc X.~Ou and J.~Xia}, {\em Super{DC}: superfast divide-and-conquer eigenvalue
  decomposition with improved stability for rank-structured matrices}, SIAM J.
  Sci. Comput., 44 (2022), pp.~A3041--A3066.

\bibitem{pan2012structured}
{\sc V.~Y. Pan}, {\em Structured matrices and polynomials: unified superfast
  algorithms}, Springer Science \& Business Media, 2012.

\bibitem{peaceman1955numerical}
{\sc D.~W. Peaceman and H.~H. Rachford}, {\em The numerical solution of
  parabolic and elliptic differential equations}, J. Soc. Ind. Appl. Math., 3
  (1955), pp.~28--41.

\bibitem{Penzl2000}
{\sc T.~Penzl}, {\em Eigenvalue decay bounds for solutions of {L}yapunov
  equations: the symmetric case}, Systems Control Lett., 40 (2000),
  pp.~139--144.

\bibitem{rubin2022bounding}
{\sc D.~Rubin, A.~Townsend, and H.~Wilber}, {\em Bounding {Z}olotarev numbers
  using {F}aber rational functions}, Const. Approx., 56 (2022), pp.~207--232.

\bibitem{sabino2007solution}
{\sc J.~Sabino}, {\em Solution of large-scale {L}yapunov equations via the
  block modified {S}mith method}, PhD thesis, Rice University, 2007.

\bibitem{townsend2018singular}
{\sc A.~Townsend and H.~Wilber}, {\em On the singular values of matrices with
  high displacement rank}, Linear Algebra Appl., 548 (2018), pp.~19--41.

\bibitem{Tyrtyshnikov1996}
{\sc E.~Tyrtyshnikov}, {\em Mosaic-skeleton approximations}, Calcolo, 33
  (1996), pp.~47--57.

\bibitem{wachspress1962optimum}
{\sc E.~L. Wachspress}, {\em Optimum alternating-direction-implicit iteration
  parameters for a model problem}, J. Soc. for Ind. and Appl. Math., 10 (1962),
  pp.~339--350.

\bibitem{wilber2021computing}
{\sc H.~D. Wilber}, {\em Computing numerically with rational functions}, PhD
  thesis, Cornell University, 2021.

\bibitem{xia2010fast}
{\sc J.~Xia, S.~Chandrasekaran, M.~Gu, and X.~S. Li}, {\em Fast algorithms for
  hierarchically semiseparable matrices}, Numer. Lin. Alg. Appl., 17 (2010),
  pp.~953--976.

\bibitem{xia2012superfast}
{\sc J.~Xia, Y.~Xi, and M.~Gu}, {\em A superfast structured solver for
  {T}oeplitz linear systems via randomized sampling}, SIAM J. Matrix Anal.
  Appl., 33 (2012), pp.~837--858.

\end{thebibliography}

\end{document}